\documentclass[12pt, reqno]{amsart}
\usepackage{amsmath}
\usepackage{amssymb}
\usepackage{epsfig}
\usepackage{graphicx}
\usepackage{color}
\usepackage{fullpage}
\definecolor{shadecolor}{gray}{0.875}
\usepackage{amscd}

\numberwithin{equation}{section}

\input xy
\xyoption{all}

\calclayout
\allowdisplaybreaks[3]

\theoremstyle{plain}
\newtheorem{prop}{Proposition}[section]

\newtheorem{theo}[prop]{Theorem}
\newtheorem{coro}[prop]{Corollary}

\newtheorem{lemm}[prop]{Lemma}

\theoremstyle{definition}
\newtheorem{defi}[prop]{Definition}

\newtheorem{rema}[prop]{Remark}

\newtheorem{exam}[prop]{Example}

\def\bR{{\mathbb R}}

\def\Eff{\overline{\mathrm{Eff}}}
\def\Pic{\mathrm{Pic}}

\def\Mor{\mathrm{Mor}}

\def\Pic{\mathrm{Pic}}

\makeatother
\makeatletter

\author{Brian Lehmann}
\address{Department of Mathematics \\
Boston College  \\
Chestnut Hill, MA \, \, 02467}
\email{lehmannb@bc.edu}

\author{Sho Tanimoto}
\address{Department of Mathematics, Faculty of Science, Kumamoto University, Kurokami 2-39-1 Kumamoto 860-8555 Japan}
\address{Priority Organization for Innovation and Excellence, Kumamoto University}
\email{stanimoto@kumamoto-u.ac.jp}

\title[Examples]{Rational curves on prime Fano threefolds of index $1$}

\begin{document}
\date{\today}

\begin{abstract}
We study the moduli spaces of rational curves on prime Fano threefolds of index $1$.  For general threefolds of most genera we compute the dimension and the number of irreducible components of these moduli spaces.  Our results confirm Geometric Manin's Conjecture in these examples and show the enumerativity of certain Gromov-Witten invariants.
\end{abstract}

\maketitle

\section{Introduction}

Let $X$ be a smooth complex Fano threefold of Picard rank $1$.  The lines and conics on $X$ play an essential role in the classification theory (\cite{iskov}) and have been extensively studied.  In this paper we will study rational curves of higher degree.  Our main theorem describes all components of $\Mor(\mathbb{P}^{1},X)$ for ``most'' Fano threefolds of Picard rank $1$ and index $1$.

\begin{theo} \label{theo: maintheorem}
Let $X$ be a smooth Fano threefold of Picard rank $1$ and index $1$ such that $-K_{X}$ is very ample and $X$ has degree $4 \leq H^{3} \leq 18$.
Assume that $X$ is general in moduli.  For any $d \geq 3$, the components of $\mathrm{Mor}(\mathbb P^1, X)$ parametrizing curves of anticanonical degree $d$ are described by the following list: 
\begin{itemize}
\item for every $d \geq 3$, there is one component parametrizing a dominant family of degree $d$ rational curves;
\item for every $d \geq 3$, there is one component parametrizing degree $d$ covers of lines;
\item when $d \geq 3$ is even, there is one component parametrizing degree $d/2$ covers of conics.
\end{itemize}
\end{theo}

This result naturally complements \cite[Theorem 7.9]{LT17} which classifies all components of $\Mor(\mathbb{P}^{1},X)$ for ``most'' Fano threefolds of Picard rank $1$ and index $2$.  (The higher index cases are of course well understood.)

Our approach to Theorem \ref{theo: maintheorem} is motivated by Manin's Conjecture.  Manin's Conjecture predicts that the asymptotic behavior of components of $\Mor(\mathbb{P}^{1},X)$ as the degree increases is controlled by two geometric invariants $a, b$ (first appearing in this context in \cite{BM}).  Thus the first step in understanding components of $\Mor(\mathbb{P}^{1},X)$ is to completely understand the behavior of the $a, b$-invariants for subvarieties and covers of $X$.  This goal is accomplished in Sections \ref{sec: geometricinvariants} -- \ref{sec: a-covers}.  The main input is an explicit analysis of the possible singularity types for hypersurfaces in $X$ of low degree.  We then use the geometric theory of the $a,b$-invariants which has been developed in the series of papers \cite{HTT15}, \cite{LTT14}, \cite{HJ16}, \cite{LT16}, \cite{LT17}, \cite{Sen17}, \cite{Sen17b}, \cite{LST18}.

The proof of Theorem \ref{theo: maintheorem} runs by induction on degree.  The induction follows the same general framework developed by \cite{HRS04} via breaking and gluing rational curves.  However, it seems difficult to apply the techniques of \cite{HRS04} directly due to the existence of components of $\Mor(\mathbb{P}^{1},X)$ with higher than expected dimension.  Instead, we systematically use the $a,b$-invariants to reorganize the argument.

As demonstrated by \cite[Theorem 1.1]{LT17} the  $a$-invariant can be used to completely classify components of $\Mor(\mathbb{P}^{1},X)$ of larger than expected dimension.  This allows us to focus our attention on dominant families of rational curves.  In order to remain in this framework while breaking curves, we must show that a free curve can be deformed to a chain of free curves.  We call this result the ``Movable Bend and Break'' lemma; the proof relies on results of \cite{Shen10} and is modeled after ideas of \cite{Testa05}.

\begin{theo}{\textnormal{(Movable Bend and Break lemma)}}
Let $X$ be a smooth Fano threefold of Picard rank $1$ and index $1$ such that $-K_{X}$ is very ample and $X$ has degree $4 \leq H^{3} \leq 18$.
Assume that $X$ is general in its moduli.
Fix $d \geq 4$ and let $M$ be a component of $\overline{M}_{0,0}(X, d)$ parametrizing a dominant family of maps which are generically birational maps from $\mathbb{P}^{1}$ to its image. Then $M$ contains a codimension $1$ locus representing stable maps which are a union of two free curves of smaller degree.
\end{theo}

Once this result is known and the analysis of $a,b$-invariants is completed, then the description of all components of the moduli space follows from a classification of low-degree components and the general theory developed by \cite{LT17}.  In fact the induction argument only begins in degree $\geq 5$, so we need to handle the cases of cubic and quartic curves separately via explicit geometric arguments (see Section \ref{sec: lowdegreecurves}).

Theorem \ref{theo: maintheorem} verifies Geometric Manin's Conjecture over the ground field $\mathbb{C}$ for this class of varieties $X$.  (In other words, it allows us to compute a counting function for sections of the trivial family $X \times \mathbb{P}^{1}$ over $\mathbb{P}^{1}$.)  In Section \ref{sec: GM} we recall the definition of the counting function for rational curves $\overline{N}(X,q,d)$ constructed by \cite{LT17}.  Theorem \ref{theo: maintheorem} verifies the expected asymptotic formula:

\begin{coro}
Let $X$ be a smooth Fano threefold of Picard rank $1$ and index $1$ such that $-K_{X}$ is very ample and $X$ has degree $4 \leq H^{3} \leq 18$.  Assume that $X$ is general in moduli.  Then
\begin{equation*}
\overline{N}(X,q,d) \sim \frac{q^{3}}{1-q^{-1}} q^{d}
\end{equation*}
as $d$ goes to $\infty$.
\end{coro}

Finally, in Section \ref{sec: gw} we prove that certain Gromov-Witten invariants on $X$ are enumerative.  We focus on the invariants $\langle [pt]^{n} \rangle_{0,n}^{X,2n}$ computing Gromov-Witten invariants with point insertions.  We show that the only contributions to this invariant come from the ``main'' component parametrizing very free curves and thus deduce that the invariants represent actual curve counts.

\begin{exam}[\cite{GOLYSHEV}]
If $X$ is a general quartic threefold then the number of rational quartic curves on $X$ passing through two general points is $1145232$.  If $X$ is a general Fano threefold of degree $18$ then there are only $27$ such curves.
\end{exam}

\noindent
{\bf Acknowledgements:}
The authors would like to thank Brendan Hassett and Yuri Tschinkel for useful discussions and answering our questions about the Abel-Jacobi map.
They would also like to thank Lars Halvard Halle and Mingmin Shen for their help regarding Hilbert-flag schemes.
They thank Dave Anderson for some suggestions concerning homogeneous varieties.  They are grateful to Eric Riedl and Ivan Cheltsov for several helpful conversations.  They also thank Damiano Testa for useful suggestions. We would like to thank Brendan Hassett and Marta Pieropan for comments on an early draft of this paper, and in particular Brendan for suggesting \cite{KS04} as a reference for low degree rational curves on the degree $22$ Fano $3$-fold.  We thank anonymous referees for careful reading, detailed suggestions, and suggesting we consider Gromov-Witten invariants, and Qile Chen for a useful conversation about them.
Brian Lehmann is supported by NSF grant 1600875.  Sho Tanimoto is partially supported by Lars Hesselholt's Niels Bohr professorship and by MEXT Japan, Leading Initiative for Excellent Young Researchers (LEADER), Inamori Foundation, and JSPS KAKENHI Early-Career Scientists Grant numbers 19K14512.

\section{Preliminaries}
\label{sec: pre}

Throughout we work over $\mathbb{C}$.  Varieties are assumed to be irreducible and reduced.  
The following lemma is extremely useful and is used freely throughout the paper.

\begin{lemm}
\label{lemm: glued components}
Let $X, Y, Z$ be schemes of finite type over $\mathbb C$ with morphisms $f: X\rightarrow Y$ and $g: Z \rightarrow Y$. Suppose that (i) $X, Y, Z$ are irreducible, (ii) $f, g$ are dominant and flat, and (iii) a general fiber of $g$ is purely $d$-dimensional. Let $N \subset X \times_Y Z$ be an irreducible component. Then the natural map $N \rightarrow X$ is dominant and $N$ has dimension $\dim X + d$. Furthermore, assume that a general fiber of $g$ is also irreducible. Then $X \times_Y Z$ is irreducible.
\end{lemm}


\subsection{Classification of smooth Fano threefolds of index $1$ and Picard rank $1$}

\begin{theo}[\cite{iskov}]
\label{theo: classification}
Let $X$ be a smooth Fano threefold of index $1$ with $\mathrm{Pic} (X) = \mathbb Z H$ for an ample divisor $H$.
Then $X$ is one of the following threefolds:
\begin{enumerate}
\item when $H^3 = 22$, $X$ is the zero locus of three sections of the rank $3$ bundle $\bigwedge^2 \mathcal U^\vee$ where $\mathcal U$ is the universal subbundle on $\mathrm{Gr}(3, 7)$;
\item when $H^3 = 18$, $X$ is a section of the homogeneous space $G_2/P \subset \mathbb P^{13}$ by a linear subspace of codimension $2$;
\item when $H^3 = 16$, $X$ is a section of the symplectic Lagrangian Grassmannian $\mathrm{LGr}(3, 6) \subset \mathbb P^{13}$ by a linear subspace of codimension $3$;
\item when $H^3 = 14$, $X$ is a section of $\mathrm{Gr}(2, 6) \subset \mathbb P^{14}$ by a linear subspace of codimension $5$;
\item when $H^3 = 12$, $X$ is a section of a connected component of the orthogonal Lagrangian Grassmannian $\mathrm{OGr_+(5, 10)}\subset \mathbb P^{15}$ by a linear subspace of codimension $7$;
\item when $H^3 = 10$, $X$ is either
\begin{enumerate}
\item a section of $\mathrm{Gr}(2, 5)\subset \mathbb P^9$ by a linear subspace of codimension $2$ and a quadric, or
\item the double cover of a section of $\mathrm{Gr}(2, 5)\subset \mathbb P^9$ by a linear subspace of codimension $3$ branched in a divisor in the anticanonical system;
\end{enumerate}
\item when $H^3=8$, $X$ is a complete intersection of three quadrics in $\mathbb P^6$;
\item when $H^3 = 6$, $X$ is a complete intersection of a quadric and cubic in $\mathbb P^5$;
\item when $H^3 =4$, $X$ is either
\begin{enumerate}
\item a quartic hypersurface in $\mathbb P^4$, or
\item the double cover of a smooth quadric branched in an intersection with a quartic;
\end{enumerate}
\item when $H^3=2$, $X$ is the double cover of $\mathbb P^3$ branched in a sextic surface
\end{enumerate}
\end{theo}

In this paper, we focus on Fano threefolds $X$ such that $-K_X$ is very ample and the degree is $< 22$.  This excludes cases 1, 6(b), 9(b), and 10.  We emphasize that in the succeeding sections the precise assumptions on $X$ will vary. 

The genus $g$ of a Fano threefold of Picard rank $1$ and index $1$ is determined by the degree via the equation $d = 2g-2$.  Each threefold such that $-K_{X}$ is very ample has a natural embedding into $\mathbb{P}^{g+1}$.  We will use $H$ to represent either the hyperplane class on $\mathbb{P}^{g+1}$ or its restriction to $X$ depending on the circumstance.

\section{Hyperplane sections of Fano threefolds}
\label{sec: hyperplane}

In our analysis of a Fano threefold $X$ it will be crucial for us to understand the singularities of elements of $|-K_{X}|$.  In this section, we record for later use some properties about these divisors using the classification.  We will mostly consider the case  when $-K_{X}$ is very ample and the degree is $\leq 8$.  

\begin{lemm}
\label{lemm: normality for H low degree}
Let $X$ be a smooth Fano threefold of Picard rank $1$ and index $1$ such that $-K_{X}$ is very ample and $X$ has degree $4 \leq H^{3} \leq 8$.  Then any member of $|H|$ is integral and normal.
\end{lemm}
\begin{proof}
In these cases $X$ is described as a complete intersection.  We give the proof when $H^3 = 8$; the other cases are similar.

When $H^3 = 8$, $X$ is a complete intersection of three quadrics $Q_0, Q_1, Q_2$ in $\mathbb P^6$. Choose a hyperplane $H$ in $\mathbb P^6$. Since $H$ is a generator for $\Pic (X)$, it is clear that $Y = X\cap H$ is reduced and irreducible. Thus our goal is to prove that $Y$ is normal. It suffices to show that $Y$ has only isolated singularities. Suppose that $Y$ is singular along a curve $C$. Without loss of generality, we may assume that $H$ is defined by $x_0=0$. Along $C$ the following matrix has rank $\leq 3$:
\[
\begin{pmatrix}
1 & 0&0&0&0&0&0\\
\frac{\partial Q_0}{\partial x_0} & \frac{\partial Q_0}{\partial x_1}& \frac{\partial Q_0}{\partial x_2}& \frac{\partial Q_0}{\partial x_3}& \frac{\partial Q_0}{\partial x_4} & \frac{\partial Q_0}{\partial x_5}& \frac{\partial Q_0}{\partial x_6}\\
\frac{\partial Q_1}{\partial x_0} & \frac{\partial Q_1}{\partial x_1}& \frac{\partial Q_1}{\partial x_2}& \frac{\partial Q_1}{\partial x_3}& \frac{\partial Q_1}{\partial x_4} & \frac{\partial Q_1}{\partial x_5}& \frac{\partial Q_1}{\partial x_6}\\
\frac{\partial Q_2}{\partial x_0} & \frac{\partial Q_2}{\partial x_1}& \frac{\partial Q_2}{\partial x_2}& \frac{\partial Q_2}{\partial x_3}& \frac{\partial Q_2}{\partial x_4} & \frac{\partial Q_2}{\partial x_5}& \frac{\partial Q_2}{\partial x_6}
\end{pmatrix}.
\]
However, since $X$ is smooth, the following matrix has rank at least $2$ (and hence exactly equal to $2$) along $C$:
\[
\begin{pmatrix}
\frac{\partial Q_0}{\partial x_1}& \frac{\partial Q_0}{\partial x_2}& \frac{\partial Q_0}{\partial x_3}& \frac{\partial Q_0}{\partial x_4} & \frac{\partial Q_0}{\partial x_5}& \frac{\partial Q_0}{\partial x_6}\\
\frac{\partial Q_1}{\partial x_1}& \frac{\partial Q_1}{\partial x_2}& \frac{\partial Q_1}{\partial x_3}& \frac{\partial Q_1}{\partial x_4} & \frac{\partial Q_1}{\partial x_5}& \frac{\partial Q_1}{\partial x_6}\\
\frac{\partial Q_2}{\partial x_1}& \frac{\partial Q_2}{\partial x_2}& \frac{\partial Q_2}{\partial x_3}& \frac{\partial Q_2}{\partial x_4} & \frac{\partial Q_2}{\partial x_5}& \frac{\partial Q_2}{\partial x_6}
\end{pmatrix}.
\]
Choose two columns such that the corresponding $2 \times 2$ minors do not all identically vanish along $C$.  The three minors $M_{12}, M_{13}, M_{23}$ define sections of the restriction of $\mathcal{O}(2)$ to $C$.  Note that
\[
M_{23}\frac{\partial Q_0}{\partial x_i} - M_{13} \frac{\partial Q_1}{\partial x_i}+ M_{12}\frac{\partial Q_2}{\partial x_i}
\]
vanishes identically along $C$ for any $1\leq i \leq 6$.
Then consider the following section of $\mathcal{O}(3)|_{C}$:
\[
M_{23}\frac{\partial Q_0}{\partial x_0} - M_{13} \frac{\partial Q_1}{\partial x_0} + M_{12} \frac{\partial Q_2}{\partial x_0}  
\]
This is an effective divisor on $C$ so it has to have a zero at some $y \in C$. This means that $X$ is singular at $y$, a contradiction.  
\end{proof}

\begin{lemm}
\label{lemm: noncanonicallocus_lowdegree}
Let $X$ be a smooth Fano threefold of Picard rank $1$ and index $1$ such that $-K_{X}$ is very ample and $X$ has degree $4 \leq H^{3} \leq 8$.  
Assume that $X$ is general in its moduli.  Let $|H| = (\mathbb P^{g+1})^*$ be the complete linear system of the hyperplane class $H$. Let $T \subset |H|$ be the locus of $S \in |H|$ such that $S$ is not canonical. Then $T$ has codimension at least $4$ in $|H|$.
\end{lemm}
\begin{proof}
Consider the flag Hilbert scheme $M$ of triples $(X,H,p)$ where $X$ is a smooth degree $d$ Fano threefold, $H$ is a hyperplane section of $X$ and $p$ is a singular point of $H$.  Recall that the projection onto the Hilbert scheme of pairs $(X,H)$ such that $H$ is singular is generically finite since by Lemma \ref{lemm: normality for H low degree} any hyperplane section can only have finitely many singularities.  It suffices to show that the sublocus $T$ consisting of those triples such that $H$ has non-canonical singularities at $p$ has codimension $\geq 3$ in $M$.  More precisely, we will show that the fiber $T_{p}$ over any point $p$ has codimension $\geq 3$ in the fiber $M_{p}$.

We write the proof carefully in the case when $X$ has degree $8$.  Then $X$ is the intersection of three smooth quadrics $Q_{i}$.  We may choose local coordinates $\{x_{j} \}_{j=1}^{6}$ in the ambient projective space so that $p$ is the point $0$ in $\mathbb{A}^{6}$ and the quadric $Q_{i}$ has equation $x_{i} + q_{i}$ for some homogeneous quadratic $q_{i}$ in the coordinates.  Consider the local analytic coordinate change
\begin{align*}
x_{i}' & = x_{i} + q_{i}(x)  & i=1,2,3 \\
x_{i}' & = x_{i}    & i=4,5,6
\end{align*}
Inverting, we see that
\begin{align*}
x_{i} & = x_{i}' - q_{i}(x') + \ldots &  i=1,2,3 \\
x_{i} & = x_{i}' & i=4,5,6
\end{align*}
The linear equation for $H$ must have the form $b_{1}x_{1} + b_{2}x_{2} + b_{3}x_{3}$ to ensure that the intersection with $X$ is singular.  Applying the analytic coordinate change and restricting $H$ to the locus where $x_{1}' = x_{2}' = x_{3}' = 0$ we see that the singularity at $p$ is locally analytically isomorphic to the hypersurface singularity
\begin{equation*}
b_{1} q_{1}(0,0,0,x_{4}',x_{5}',x_{6}') + b_{2}q_{2}(0,0,0,x_{4}',x_{5}',x_{6}') + b_{3}q_{3}(0,0,0,x_{4}',x_{5}',x_{6}') + \ldots 
\end{equation*}
in $\mathbb{C}^{3}$.  If the corank of this quadratic is $0$ or $1$ then $H$ is guaranteed to have canonical singularities at $p$ (of type $A_{m}$).  But the locus of quadratic equations of corank $\geq 2$ has codimension $\geq 3$.
The other two cases are essentially the same.
\end{proof}

In the degree $4$ case we will also need to understand singularities for elements of $|-2K_{X}|$.

\begin{lemm}
\label{lemm: normality for 2H}
Let $X$ be a smooth quartic threefold.
Suppose that $S \in |2H|$ is the restriction of a quadric of rank at least $4$. 
Then $S$ is integral and normal.
\end{lemm}

\begin{proof}
Suppose that $X$ is defined by a homogeneous polynomial $f_4$. To prove our statement, it suffices to show that $S$ has only isolated singularities.

Suppose that $S$ is singular along a curve $C$.
Along $C$, the following matrix has rank $1$:
\[
\begin{pmatrix}
\frac{\partial Q}{\partial x_0} &\frac{\partial Q}{\partial x_1} & \frac{\partial Q}{\partial x_2} &\frac{\partial Q}{\partial x_3} &\frac{\partial Q}{\partial x_4}\\
\frac{\partial f_4}{\partial x_0} &\frac{\partial f_4}{\partial x_1} &\frac{\partial f_4}{\partial x_2} &\frac{\partial f_4}{\partial x_3} & \frac{\partial f_4}{\partial x_4}
\end{pmatrix}
\]
In particular, for each $x\in C$, there exists a unique point $(\lambda(x):\mu (x)) \in \mathbb P(3,1)$ such that
\[
\lambda(x)\frac{\partial Q}{\partial x_i}(x) + \mu (x)\frac{\partial f_4}{\partial x_i}(x) = 0
\]
for any $0\leq i \leq 4$.
Since $Q$ is a quadric with at most one singularity, the morphism $$C \ni x \mapsto  (\lambda(x): \mu(x)) \in \mathbb P(3,1)$$ is dominant. Thus for some $y \in C$, we have $(\lambda(y): \mu(y))=(0:1)$.
This implies that $X$ is singular at $y$, a contradiction.
\end{proof}

\begin{lemm}
\label{lemm: nonnormality for 2H}
Let $X$ be a smooth quartic threefold.
Suppose that $S \in |2H|$ is an integral surface.  
\begin{itemize}
\item Suppose there is no quadric $Q$ which is singular along a line $l$ contained in $X$ such that $Q|_{X} = S$.  Then $S$ is normal.
\item Suppose there is a quadric $Q$ which is singular along a line $l$ contained in $X$ such that $Q|_{X} = S$. Then the blow-up of $S$ along this line is normal.
\end{itemize}
\end{lemm}

\begin{proof}
Note that the restriction of any quadric $Q$ of rank $\leq 2$ to $X$ will not be integral.  Thus $S$ can only be the restriction of a quadric of rank $\geq 3$.  The first statement follows from the arguments of Lemma \ref{lemm: normality for 2H} and we focus on the second.  When the singular locus of the quadric is contained in $X$, $S$ is not normal and (again by the argument of Lemma \ref{lemm: normality for 2H}) the singular locus is the union of $l$ with possibly finitely many points.  We may suppose without loss of generality that the quadric is $x_{0}^{2} + x_{1}^{2} + x_{2}^{2}$ and that the equation of the quartic is
\begin{equation*}
f = x_{0}P_{0}(x_{3},x_{4}) + x_{1}P_{1}(x_{3},x_{4}) + x_{2}P_{2}(x_{3},x_{4}) + \ldots
\end{equation*}
where we omit the higher order terms in $x_{0},x_{1},x_{2}$.  Since the quartic is smooth along the line, there is no point on the line where all of the $P_{i}$ simultaneously vanish.

Consider the blow-up of $\mathbb{P}^{4}$ along $l$; we want to check for singularities on the strict transform of $S$ lying above $l$.  We will argue locally using charts.  Assigning variables $u,v,w$ on $\mathbb{P}^{2}$ to correspond with $x_{0},x_{1},x_{2}$, we first consider the chart given by variables $\frac{x_{0}}{x_{4}}, \frac{x_{3}}{x_{4}}, \frac{v}{u}, \frac{w}{u}$.  On this chart the strict transform of the quadric has the equation
\begin{equation*}
1 + (v/u)^{2} + (w/u)^{2}
\end{equation*}
so that the Jacobian is $(0,0,2v/u, 2w/u)$.  The strict transform of the quartic has equation
\begin{equation*}
\widetilde{f} = P_{0}(x_{3}/x_{4},1) + (v/u) P_{1}(x_{3}/x_{4},1) + (w/u) P_{2}(x_{3}/x_{4},1) + (x_{0}/x_{4}) \cdot g
\end{equation*}
for some polynomial $g$.
For points on the exceptional divisor $x_{0}/x_{4}=0$, we have
\begin{align*}
\frac{d\widetilde{f}}{d(x_{3}/x_{4})} = \frac{dP_{0}}{d(x_{3}/x_{4})} + (v/u) \frac{dP_{1}}{d(x_{3}/x_{4})} + (w/u) \frac{dP_{2}}{d(x_{3}/x_{4})} \\
 \frac{d\widetilde{f}}{d(v/u)} = P_{1}(x_{3}/x_{4},1) \qquad \qquad \frac{d\widetilde{f}}{d(w/u)} = P_{2}(x_{3}/x_{4},1).
\end{align*}
If the intersection has a singularity along the exceptional divisor, then at this point there must be a constant $C$ such that 
\begin{equation*}
(v/u) = C P_{1}(x_{3}/x_{4},1) \qquad \qquad (w/u) = C P_{2}(x_{3}/x_{4},1).
\end{equation*}
Substituting into the equations for the quadric and quartic, we find that
\begin{align*}
1+C^{2}P_{1}^{2} + C^{2}P_{2}^{2} & = 0 \\
P_{0} + CP_{1}^{2} + CP_{2}^{2} & = 0
\end{align*}
The first equation shows that $C \neq 0$.  Furthermore, we know that
\begin{equation*}
\frac{d\widetilde{f}}{d(x_{3}/x_{4})} = \frac{d}{d(x_{3}/x_{4})} \left( P_{0} + \frac{C}{2}P_{1}^{2} + \frac{C}{2}P_{2}^{2} \right)
\end{equation*}
must vanish at any singular point.  Assume that there is a singular point above every point of $l$.  Comparing the previous equations, we see that $P_{0}$ must be constant along this chart of $l$.  Arguing similarly along all the charts, we obtain a contradiction to the smoothness of the quartic along the line.  Thus the singularities of the strict transform of $S$ are isolated.
\end{proof}

We will need one result that also holds for Fano threefolds whose anticanonical divisor is very ample.

\begin{lemm}
\label{lemm: normality for H high degree}
Let $X$ be a Fano threefold of Picard rank $1$ and index $1$ with $-K_X$ very ample.  The sublocus of $\mathbb{P}^{g+1}$ parametrizing non-normal hyperplane sections has codimension $\geq 3$.
\end{lemm}

\begin{proof}
For each $p \in X \subset \mathbb P^{g+1}$, hyperplane sections $H \in (\mathbb P^{g+1})^*$ which are singular at $p$ form a $\mathbb P^{g-3}$ in $(\mathbb P^{g+1})^*$. This gives us a $\mathbb P^{g-3}$-bundle $\mathcal H$ over $X$ which is a subvariety of $X \times  \mathbb P^{g+1}$.  The projection $\pi: \mathcal H \rightarrow  \mathbb P^{g+1}$ is generically finite since none of our Fano threefolds are ruled in projective subspaces.
The exceptional locus $E$ of $\pi$ is at most $(g-1)$-dimensional, thus the image $\pi(E)$ is at most $(g-2)$-dimensional.
\end{proof}

\section{Geometric invariants in Manin's conjecture}
\label{sec: geometricinvariants}

In \cite{LT17}, we use the geometric invariants in Manin's conjecture to systematically study the parameter space of rational curves $\mathrm{Mor}(\mathbb P^1, X)$. We recall these invariants in this section.
These invariants also appeared in the study of cylinders in \cite{CPW16}.
We denote the cone of pseudoeffective divisors of $X$ by $\overline{\mathrm{Eff}}^1(X)$.

\begin{defi}{\cite[Definition 2.2]{HTT15}}
Let $X$ be a smooth projective variety and let $L$ be a big and nef $\mathbb Q$-Cartier divisor on $X$.
The Fujita invariant is
$$
a(X, L) := \min \{ t\in \bR \mid K_{X} + tL \in \Eff^{1}(X) \}.
$$
If $L$ is nef but not big, we set $a(X, L) =+ \infty$.  If $X$ is singular, then we define this invariant as $a(\widetilde{X},\phi^{*}L)$ where $\phi: \widetilde{X} \to X$ is a resolution of singularities.
(The value is independent of the choice of resolution by \cite[Proposition 2.7]{HTT15}.)  Note that $a(X, L) >0$ if and only if $X$ is uniruled by \cite{BDPP}.
\end{defi}

While studying the $a$-invariant one can often restrict attention to the following special class of varieties.

\begin{defi}
Let $(X,L)$ be a pair consisting of a smooth projective variety $X$ and a big and nef $\mathbb{Q}$-divisor $L$ such that $a(X,L) > 0$.  We say that $(X,L)$ is adjoint rigid, or equivalently that the adjoint divisor $K_{X} + a(X,L)L$ is rigid, if $K_{X} + a(X,L)L$ has Iitaka dimension $0$.

If $(X,L)$ is a pair consisting of a singular projective variety $X$ and a big and nef $\mathbb{Q}$-divisor $L$ such that $a(X,L)>0$, we say that $(X,L)$ is adjoint rigid if for some resolution $\phi: X' \to X$ the pair $(X',\phi^{*}L)$ is adjoint rigid.
\end{defi}

In \cite{LT17}, we showed that the $a$-invariant controls the dimension of moduli spaces of rational curves:

\begin{theo}[\cite{LT17} Theorem 1.1] \label{intro:expecteddim}
Let $X$ be a smooth projective weak Fano variety and set $L = -K_{X}$.  Let $V \subsetneq X$ be the proper closed subset which is the union of all subvarieties $Y$ such that $a(Y,L|_{Y}) > a(X,L)$.  Then any component of $\mathrm{Mor}(\mathbb{P}^{1},X)$ parametrizing a curve not contained in $V$ will have the expected dimension and will parametrize a dominant family of rational curves.
\end{theo}

The fact that the set $V$ in the statement is a proper closed subset of $X$ is a result of \cite{HJ16}. 

\begin{defi}{\cite[Definition 2.8]{HTT15}}
Let $X$ be a smooth uniruled projective variety and let $L$ be a big and nef $\mathbb Q$-Cartier divisor on $X$.
We define $b(X, L)$ to be the codimension of the minimal face of $\Eff^{1}(X)$ containing the class of $K_{X} + a(X, L)L$. 
When $X$ is singular, we define this invariant as $b(\widetilde{X},\phi^{*}L)$ where $\phi: \widetilde{X} \to X$ is a resolution of singularities.
(The value is again independent of the choice of resolution by \cite[Proposition 2.10]{HTT15}.)
\end{defi}

This invariant measures the Picard rank of a variety associated to $(X,L)$ via the minimal model program (see \cite[Corollary 3.9 and Lemma 3.10]{LTT14}). Its behavior in families has been studied in \cite{LT16} and \cite{Sen17}.  This invariant will not play a large role in this paper; we will only use it to explain the formulation of Manin's Conjecture in Section \ref{sec: GM}.

Finally, we record for later use a basic lemma about the $a$-value of surfaces.

\begin{lemm}
\label{lemm: spectrumforsurfaces}
Let $S$ be a smooth projective surface and let $L$ be a big and nef Cartier divisor on $S$.
Then $a(S, L)$ takes the form of $2/n$ or $3/n$ where $n$ is a positive integer.
\end{lemm}
\begin{proof}
We apply the arguments of \cite[Proposition 4.6]{LTT14}.
If the Iitaka dimension of the adjoint divisor $a(S, L)L + K_S$ is positive,
then let $F$ be a fiber of the Iitaka fibration associated to this adjoint divisor.  In this situation $a(S,L) = a(F,L)$.
Since a general fiber is isomorphic to $\mathbb P^1$ we have $a(S, L)L \cdot F = 2$.
Thus our assertion follows in this case.

Suppose that $a(S, L)L + K_S$ is rigid.
By applying the MMP, one can obtain a birational morphism $\beta: S \rightarrow S'$ such that $g: S' \rightarrow Z$ is a Mori fiber space from a smooth surface and $a(S, L)\beta_*L + K_{S'} \equiv 0$.
If the Mori fiber space $g: S' \rightarrow Z$ is a $\mathbb P^1$-bundle, then we have
$a(S, L)\beta_*L \cdot F = 2$ where $F$ is a fiber of $g$.
If $S'$ is $\mathbb P^2$, then we have $a(S, L)\beta_*L \cdot H = 3$ where $H$ is a hyperplane section on $\mathbb P^2$. Thus our assertion follows.
\end{proof}

\section{Manin's invariants for subvarieties}
\label{sec: Manininvariants}
Let $X$ be a smooth Fano threefold of Picard rank $1$ and index $1$ satisfying $H^3 \geq 4$, and when $H^3 = 4$ assume also that $-K_X$ is very ample.  Our goal is to classify the subvarieties of $X$ with high $a$-invariant with respect to $-K_{X}$.  Since curves are straightforward, the main issue is to understand surfaces with large $a$-invariant.

Suppose that $S \subset X$ is a surface.  \cite[Section 6.4]{LTT14} shows that $a(S,H) > 1$ if and only if $S$ is swept out by lines.  We will let $Z$ denote the union of all such surfaces.  We also must classify surfaces with $a(S,H)=1$.  When $H^3 \geq 10$, \cite{LTT14} shows that any surface $S$ with $a(S, H)=1$ is ruled by conics in $X$ and if $\phi: S' \to S$ is a resolution then the adjoint divisor $K_{S'} + \phi^{*}H|_{S}$ has Iitaka dimension $1$.  We prove this statement for the other Fano threefolds on our list:

\begin{theo}
\label{theo: weaklybalanced for Fano 3-folds}
Let $X$ be a smooth Fano threefold of Picard rank $1$, index $1$, and degree $H^{3} \geq 4$.    Furthermore if $H^{3} = 4$ assume that $-K_{X}$ is very ample.  For any surface $S\not\subset Z$ with $a(S, H) = 1$ we have that $(S,H)$ is not adjoint rigid and that $S$ is covered by conics.  
\end{theo}

\begin{rema}
One consequence of this theorem is that the $b$-invariant is uninteresting in our situation.  More precisely, if $f: Y \to X$ is any morphism that is generically finite onto its image such that $f$ does not factor rationally through $Z$ and $a(Y,f^{*}H) \geq a(X,H)$ then $b(Y,f^{*}H) = 1$.  For this reason we will not need to consider the $b$-invariant while discussing subvarieties or covers.
\end{rema}

The main tool is the following useful lemma.

\begin{lemm} \label{lemm:surfaceoptions}
Suppose that $Y$ is a normal surface and $L$ is an ample Cartier divisor on $Y$.  Then one of the following holds:
\begin{itemize}
\item $a(Y,L)<1$, or
\item $Y$ has canonical singularities and $a(Y,L)L \equiv -K_{Y}$, or
\item for a resolution $\phi: \widetilde{Y} \to Y$ the adjoint pair $K_{\widetilde{Y}} + a(Y,L)\phi^{*}L$ has Iitaka dimension $1$.
\end{itemize}
\end{lemm}

\begin{proof}
Suppose that $a(Y,L) \geq 1$.  Let $\phi: \widetilde{Y} \to Y$ be a resolution.  Suppose we run the MMP for $K_{\widetilde{Y}} + a(Y,L)\phi^{*}L$.  The first step contracts a $(-1)$-curve $C$ which is negative against this class, and since $L$ is Cartier we must have $\phi^{*}L \cdot C = 0$.  Thus the result of the contraction is a morphism $\psi: \widetilde{Y} \to \widehat{Y}$ such that $\phi$ factors through $\psi$.  Since $\widehat{Y}$ is smooth, we may replace $\widetilde{Y}$ by $\widehat{Y}$.

Repeating this process, we will eventually end at a minimal model.  There are two possible outcomes:
\begin{itemize}
\item $\kappa(K_{\widetilde{Y}} + a(Y,L)\phi^{*}L) = 1$.  This is the last case.
\item $\kappa(K_{\widetilde{Y}} + a(Y,L)\phi^{*}L) = 0$.  In this case the result of the MMP is a smooth weak del Pezzo surface, and arguing as above we may replace $\widetilde{Y}$ by this surface.  After this change we have $-K_{\widetilde{Y}} \equiv a(Y,L)\phi^{*}L$.  Since $Y$ is normal, the map $\phi: \widetilde{Y} \to Y$ is a contraction of an extremal face of the Mori cone, so that $Y$ must have canonical singularities.  Thus we may push forward the numerical equality of divisors on $\widetilde{Y}$ to obtain $-K_{Y} \equiv a(Y,L)L$.
\end{itemize}
\end{proof}

By \cite[Theorem 3.7]{LT16}, any surface with $a(S,-K_{X}) \geq 1$ is swept out by lines, conics, or cubics.  In fact the case of cubics is not interesting for us:

\begin{lemm}
Suppose that $Y$ is a normal surface and $L$ is an ample Cartier divisor on $Y$ such that the smallest $L$-degree of a rational curve through a general point of $Y$ is $3$.  If $a(Y,L) \geq 1$ then $Y \cong \mathbb{P}^{2}$ and $L = \mathcal{O}(1)$.
\end{lemm}

\begin{proof}
In this case it is impossible for the adjoint pair on a resolution to have Iitaka dimension $1$.  Thus arguing as in the proof of Lemma \ref{lemm:surfaceoptions}, we obtain a smooth weak del Pezzo surface $\widetilde{Y}$ and a birational map $\phi: \widetilde{Y} \to Y$ such that $a(Y,L)\phi^{*}L \equiv -K_{\widetilde{Y}}$.  But any weak del Pezzo besides $\mathbb{P}^{2}$ is covered by curves of anticanonical degree $2$.  So $\widetilde{Y} = \mathbb{P}^{2}$ and the claim follows.
\end{proof}

\begin{proof}[Proof of Theorem~\ref{theo: weaklybalanced for Fano 3-folds}]
When $H^3 \geq 10$, this statement is proved in \cite[Proposition 6.12]{LTT14}.
Suppose that $6 \leq H^3 \leq 8$. If $S \in |nH|$ where $n\geq 2$, then the proof of  \cite[Proposition 6.12]{LTT14} implies our assertion. So we may assume that $S \in |H|$. By Lemma~\ref{lemm: normality for H low degree}, $S$ is normal so if we have $a(S, H)=1$, then by Lemma~\ref{lemm:surfaceoptions}, we can conclude that either $S$ has canonical singularities or the adjoint divisor on a smooth model has Iitaka dimension $1$. However, if $S$ has only canonical singularities, then it follows from \cite[Proposition 2.7]{HTT15} and the adjunction formula that $a(S, H) < 1$. Thus the surface is adjoint rigid and is fibered by conics.

In the case of $H^3=4$, then we need to handle the cases of $S\in |nH|$ where $n=1, 2$. The case of $n=1$ is the same as above, so we may assume that $n=2$.  Let $S = Q\cap X \in |2H|$ be an integral divisor where $Q$ is a quadric in $\mathbb P^4$. If $S$ is normal, then we can argue as before.  By Lemmas \ref{lemm: normality for 2H} and \ref{lemm: nonnormality for 2H}, the only other case to consider is when $Q$ is singular along a line contained in $X$.  Then the strict transform of $S$ on the blow-up of $X$ along $l$ is normal.  If the strict transform of $S$ has canonical singularities then an easy adjunction argument shows its canonical divisor is pseudo-effective.  Just as before we conclude by Lemma \ref{lemm:surfaceoptions} that either $a(S,H) < 1$ or the adjoint divisor has Iitaka dimension $1$.
\end{proof}

\subsection{Surfaces with $a$-value 3/4}

By Lemma \ref{lemm: spectrumforsurfaces}, the next possible highest value for $a$-values is $a(S,L) = 3/4$.  We show that in many cases there is no surface $S$ with this $a$-value.

\begin{theo}
\label{theo: 34classification}
Suppose that $Y$ is a normal surface and that $L$ is an ample Cartier divisor on $Y$ such that $a(Y,L) = 3/4$.  Then there is a resolution $\phi: \tilde{Y} \to Y$ and a diagram
\begin{equation*}
\xymatrix{\tilde{Y} \ar[r]^{f}\ar[d]_{\phi}&  \mathbb{P}^{2}\\
Y & }
\end{equation*}
such that:
\begin{itemize}
\item $f$ is a birational map whose non-trivial fibers are chains of (-2)-curves with a single (-1)-curve on the end, and
\item $\phi^{*}L = f^{*}\mathcal{O}(4) - K_{\tilde{Y}/\mathbb{P}^{2}} = -K_{\widetilde{Y}} + f^{*}\mathcal{O}(1)$.
\end{itemize}
\end{theo}

\begin{proof}
Let $\phi: \tilde{Y} \to Y$ be any resolution.  Note that by the argument of Lemma \ref{lemm: spectrumforsurfaces} the value $a(\tilde{Y},\phi^{*}L)=3/4$ is only possible in the adjoint rigid case.

Suppose we run the MMP on $\tilde{Y}$ for the terminal pair $K_{\tilde{Y}} + \frac{3}{4}\phi^{*}L$.  The result is a sequence of contractions of $(-1)$-curves that are negative for $K_{\tilde{Y}} + \frac{3}{4}\phi^{*}L$.  Since the pair is rigid, the program terminates with a morphism $f: \tilde{Y} \to S$ where $S$ is a weak Fano variety and $\frac{3}{4}f_{*}\phi^{*}L \equiv -K_{S}$.  Note that if $S$ carries any curve $C$ with $-K_{S} \cdot C = 2$ then $f_{*}\phi^{*}L \cdot C = 8/3$, an impossibility.  We conclude that $S = \mathbb{P}^{2}$ and $f_{*}\phi^{*}L = \mathcal{O}(4)$.

We now study the steps of the MMP in more detail.  The program realizes $\widetilde{Y}$ as a sequence of blow-ups of $\mathbb{P}^{2}$.  We write this sequence as:
\begin{equation*}
\xymatrix{
\widetilde{Y} = S_{r} \ar@{>}[r]^-{f_{r}} & S_{r-1} \ar@{>}[r]^-{f_{r-1}} &S_{r-2} \ldots \ar@{>}[r]^-{f_2} &S_{1} \ar@{>}[r]^-{f_{1}}  & S_{0} = \mathbb{P}^{2}
}
\end{equation*}
We also let $\psi_{i} : \widetilde{Y} \to S_{i}$ denote the composition of blow-ups.

Each step of MMP is negative with respect to $K_{S_{i}} + \frac{3}{4} {\psi_{i}}_{*}\phi^{*}L$.  Since the pushforward of $\phi^{*}L$ is nef and Cartier, we see that $\psi_{i*}\phi^{*}L$ has intersection either $0$ or $1$ against the curve contracted by $f_{i}$.  We denote this number by $\delta_{i}$.  Note that if $\delta_{i}=0$ then $\psi_{i*}\phi^{*}L = f_{i}^{*}\psi_{i-1*}\phi^{*}L$, and if $\delta_{i} = 1$ then $\psi_{i*}\phi^{*}L = f_{i}^{*}\psi_{i-1*}\phi^{*}L - F_{i}$ where $F_{i}$ is the exceptional divisor on $S_{i}$.

Consider the preimage of a fixed point $p$ which is an exceptional center for $f$.  We may rearrange the order of the blow-ups so that the first $k$ of the blow-ups extract the preimage of the point.  Suppose that $\delta_{1} = 0$.  Then also $\delta_{2}=0$, since otherwise $\psi_{2*}\phi^{*}L$ would have negative intersection against the strict transform of $F_{1}$.  The same logic shows that all $\delta_{i} = 0$ for $1 \leq i \leq k$.  Suppose instead that $\delta_{1} = 1$.  Then we can repeat the same argument with the next step to see that if $\delta_{2} = 0$ then any further blow-up centered over $F_{2}$ will also have $\delta$-value $0$.  In sum, after possibly rearranging the order of the blow-ups, there is some $0 \leq j \leq k$ such that $\delta_{i} = 1$ for $1 \leq i \leq j$ and $\delta_{i} = 0$ for $j < i \leq k$.

Note however that since $L$ is ample on $Y$, the map $\phi: \widetilde{Y} \to Y$ will factor through any intermediate stage of the MMP satisfying $\psi_{i}^{*}\psi_{i*}\phi^{*}L = \phi^{*}L$.  Thus after replacing $\widetilde{Y}$ by a birational model, we may suppose that every $\delta_{i} = 1$.  It is then easy to see by induction that $\phi^{*}L = f^{*}\mathcal{O}(4) - K_{\widetilde{Y}/\mathbb{P}^{2}}$.

Finally, we need to verify the statement on fibers.  Maintaining the notation from above, we know that $\delta_{1} = 1$.  If there is a second blow-up, then $\psi_{2*}\phi^{*}L$ has vanishing intersection with the strict transform of $F_{1}$.  Thus, by the same logic as earlier, no further blow-ups can happen along the strict transform of $F_{1}$.  Repeating the argument inductively finishes the proof.
\end{proof}

To continue the argument, we make a temporary definition:

\begin{defi}
Let $X$ be a normal $\mathbb{Q}$-Gorenstein surface and let $L$ be a big and nef $\mathbb{Q}$-Cartier divisor on $X$.  Then we set
\begin{equation*}
\widetilde{a}(X, L) := \min \{ t\in \bR \mid K_{X} + tL \in \Eff^{1}(X) \}.
\end{equation*}
\end{defi}

The distinction with $a(X,L)$ is that we do not first pass to a resolution of $X$.  In particular, we can not expect $\widetilde{a}(X,L)$ to control rational curves in the same way.  (Note however that if $X$ has only log terminal singularities then by \cite{KM99} $\widetilde{a}(X,L)$ controls the presence of $K_{X}$-negative rational curves.) 


We now continue our study by specializing the argument based on the $L$-degree of $Y$.  

\begin{prop} \label{prop: logterminalcase}
Let $Y$ be a normal $\mathbb{Q}$-Gorenstein surface.  Suppose that $L$ is an ample Cartier divisor on $Y$ satisfying $a(Y,L) = 3/4$ and $L^{2} \geq 6$.  Then $\widetilde{a}(Y,L) > 0$.
\end{prop}

\begin{proof}
Apply Theorem \ref{theo: 34classification} to construct a resolution $\phi: \widetilde{Y} \to Y$ and a birational map $f: \widetilde{Y} \to \mathbb{P}^{2}$.  Note that $f$ can be the blow-up of at most $10$ points, and although these points can be infinitely near the fibers of $f$ have the simple description given in Theorem \ref{theo: 34classification}.

We now study which curves $C$ on $\widetilde{Y}$ can possibly satisfy $\phi^{*}L \cdot C = 0$ using the description as a blow-up of $\mathbb{P}^{2}$.  If such a curve $C$ is $f$-exceptional, then it must be a $(-2)$-curve in one of the fibers of $f$.  Otherwise $C$ is the strict transform of a curve in $\mathbb{P}^{2}$.
\begin{itemize}
\item If $C$ is the strict transform of a line $T$, then we must have blown-up four points on $T$.  Note that if we blow-up five points on a line then $L$ is no longer nef.
\item If $C$ is the strict transform of a conic $T$, then we must have blown-up eight points of $T$.
\end{itemize}
We claim that no higher degrees are possible.  Indeed, suppose $C$ has vanishing intersection against $\phi^{*}L$ and is the strict transform of an irreducible curve of degree $d$.  We let $m_{i}$ denote the multiplicity of the strict transform of $C$ at the $i$th blown-up point.  Thus
\begin{equation*}
d = K_{\widetilde{Y}} \cdot C = -3d + \sum m_{i}.
\end{equation*}
We also have $C^{2} = d^{2} - \sum m_{i}^{2}$.  Computing the arithmetic genus:
\begin{equation*}
-2 \leq (K_{\widetilde{Y}} + C) \cdot C = d + d^{2} - \sum m_{i}^{2}.
\end{equation*}
The right hand side is maximized when the sum of squares is as small as possible subject to the condition that $4d = \sum m_{i}$ and there are at most $10$ points.  Under these constraints the inequality can not be satisfied when $d > 2$.

So the possible fibers for the map $\phi$ are as follows:
\begin{enumerate}
\item A fiber can be a chain of $(-2)$-curves.
\item A fiber can be the strict transform of a line blown-up in exactly 4 points, with possibly some chains of $(-2)$-curves attached.
\item A fiber can be the strict transform of a conic blown-up in exactly 8 points, with possibly some chains of $(-2)$-curves attached.
\item A fiber can be the strict transform of two lines which are each blown-up in exactly 4 points and which meet at a non-exceptional point in $\mathbb{P}^{2}$, with possibly some chains of $(-2)$-curves attached.
\item A fiber can be the strict transform of two lines which are each blown-up in exactly 4 points and which meet at an exceptional point in $\mathbb{P}^{2}$.  In this case the strict transforms are connected by a chain of $(-2)$-curves of length up to 4 and there are possibly some chains of $(-2)$-curves attached to this structure.
\end{enumerate}
One can easily determine $\phi^{*}K_{Y}$ in each case using adjunction, but there are too many cases to enumerate here.  Often $Y$ has log canonical singularities, but this is not always the case.

The set of all fibers of $\phi$ falls into one of the following categories:
\begin{itemize}
\item a single fiber of any of the types above along with possibly some fibers of type (1),
\item up to four fibers of type (2) along with possibly some fibers of type (1),
\item a single fiber of type (3) and a single fiber of type (2),
\item a single fiber of type (4) and a single fiber of type (2),
\item a single fiber of type (5) and a single fiber of type (2).
\end{itemize}
In each of these cases we claim that $\phi^{*}K_{Y}$ has negative intersection against the strict transform of a general line on $\mathbb{P}^{2}$. This implies that $\widetilde{a}(Y,L) > 0$. Let us analyze the above claim briefly for each case.  Note that $f_{*}\phi^{*}K_{Y} = \mathcal O(-3) - f_{*}K_{\widetilde{Y}/Y}$.  Thus we will write $f_{*}K_{\widetilde{Y}/Y}$ as a weighted sum of curves in $\mathbb{P}^{2}$ and we just need to verify that the total degree is $>-3$.

\noindent
\textbf{Unique fiber: a single conic with chains of $(-2)$-curves.}
Let $C$ be the strict transform of a conic which is contracted by $\phi$. Let $P_1, \cdots, P_r$ be the points on $C$ such that chains of $(-2)$-exceptional curves are attached to them. Let $m_i= \text{the length of the chain at } P_i +1$. We know that
\[
\sum_{i=1}^r m_i \leq 10.
\]
Let $a$ be the coefficient of $C$ in $K_{\widetilde{Y}}- \phi^*K_Y$. Let $b_i$ be the coefficient of the end curve of the chain at $P_i$ in $K_{\widetilde{Y}}- \phi^*K_Y$. Then we have
\[
a = m_ib_i.
\]
Since the self intersection of $C$ is $-4$, by taking the intersection of $C$ against $K_{\widetilde{Y}}- \phi^*K_Y$, we obtain
\[
2 = -4a + \sum_{i=1}^r \frac{m_i-1}{m_i}a
\]
We would like to minimize $2a$ and show that this is greater than $-3$.
To minimize $a$ we need to maximize $\sum_{i=1}^r \frac{m_i-1}{m_i}$, and this is achieved when $r=5$ and all $m_i$'s are $2$.
In this situation, $a = -4/3$ hence our assertion follows.

\noindent
\textbf{Unique fiber: a single line with chains of $(-2)$-curves.} Similar to the single conic case.

\noindent
\textbf{Unique fiber: two lines with chains of $(-2)$-curves.}
Let $C_1, C_2$ be the strict transforms of the two lines and  denote the coefficient of $C_1, C_2$ in $K_{\widetilde{Y}}- \phi^*K_Y$ by $a, b$ respectively. First suppose that the intersection point of two lines is nonexceptional. Let $P_1, \cdots, P_r$ be the points on $C_1$ such that the chains of $(-2)$-curves are attached to these points and its length $+1$ is denoted by $m_i$. We let $Q_1, \cdots, Q_s$, $n_1, \cdots, n_s$ denote the same thing for $C_2$. By taking the intersection of $C_1$ against $K_{\widetilde{Y}}$, we obtain
\[
1 = -3a + \left(\sum_i \frac{m_i-1}{m_i}\right)a + b.
\]
Similarly for $C_2$ we have
\[
1= a -3b + \left(\sum_j \frac{n_j-1}{n_j} \right)b.
\]
We denote $\sum_i \frac{m_i-1}{m_i}$ by $\alpha$ and $\sum_j \frac{n_j-1}{n_j}$ by $\beta$ and by solving the above linear equation, one obtains
\[
a+b = \frac{\alpha+\beta - 8}{\alpha\beta - 3(\alpha+\beta)+8}
\]
Our goal here is to minimize $a+b$ and see that it is greater than $-3$. A simple calculus argument shows that if one fixes $\alpha + \beta$, then $a + b$ is minimized when $\alpha = \beta$. Under this assumption, if we vary $\alpha = \beta$, then $a+ b$ is minimized when $\alpha$ is maximum which is bounded by $5/4$ in our situation. In this situation $a+b = -11/4 > -3$. Thus our assertion follows.
When the intersection is an exceptional point, the analysis is similar.

\noindent
\textbf{Two fibers: a conic and a line.}
In this situation, for a conic one can attach at most two chains with two $(-2)$-curves, and the minimum of the coefficient of such a conic is $-2/3$.
For a line one can attach at most two chains with six $(-2)$-curves, and the minimum of the coefficient of such a line is $-2/3$.
Thus our assertion follows.

\noindent
\textbf{Two fibers: a pair of lines and a line.}
For the line disjoint from others, one can attach at most two chains with six $(-2)$-curves, so the minimum of the coefficient of such a line is $-2/3$. For the pair of lines, one can attach at most two chains with two $(-2)$-curves and the minimum value of the sum of coefficients of two lines is $-19/28$. Thus our assertion follows.

\noindent
\textbf{Two fibers: two disjoint lines.}
Two lines are disjoint, and for each line one can attach at most four chains with total six $(-2)$-curves, and the minimum of the coefficient of such a line is $-4/3$. Thus our assertion follows.

\noindent
\textbf{Three fibers: three disjoint lines.}
Three lines are disjoint, and for each line one can attach at most three chains with total five $(-2)$-curves, and the minimum of the coefficient of such a line is $-6/7$.
Thus our assertion follows.

\noindent
\textbf{Four fibers: four disjoint lines.}
To achieve this arrangement, we can't have any $-2$ curves anywhere and we get four coefficients of $-1/3$.
\end{proof}

We can now apply the classification to understand surfaces in Fano threefolds with $a$-invariant equal to $3/4$.

\begin{theo}
\label{theo: a-value=3/4}
Let $X$ be a Fano threefold of Picard rank $1$ and index $1$.
\begin{itemize}
\item Suppose that $X$ has degree $\geq 18$.  Then for every surface $Y$ on $X$ we have $a(Y,H) \neq 3/4$.
\item Suppose that $X$ has degree $10 \leq H^3 \leq 16$.  Then for every surface $Y$ on $X$ we have $a(Y,H) \neq 3/4$ except possibly for non-normal members of the linear series $|H|$.
\end{itemize}
\end{theo}

\begin{proof}
Recall that if $a(Y,H) = 3/4$ then the adjoint pair on a resolution is rigid.  Thus by the arguments of \cite[Proposition 3.6]{LT16} we know that if $(\frac{3}{4}H)^{2} \cdot Y > 9$ then $a(Y,H) \neq 3/4$.  Thus the result follows immediately when the degree of $X$ is $\geq 18$ or when $Y$ is not a hyperplane section of $X$. The only remaining cases to consider are when the degree is $10 \leq H^{3} \leq 16$ and $Y$ is a normal hyperplane section of $X$. This implies that $Y$ has isolated singularities and is Gorenstein.  By adjunction we know that $\widetilde{a}(Y,H) = 0$.  Applying Proposition \ref{prop: logterminalcase} we see that if $a(Y,H) = 3/4$ then $\widetilde{a}(Y,H) > 0$, a contradiction.
\end{proof}

\section{$a$-covers}
\label{sec: a-covers}

To study Manin's Conjecture one must analyze generically finite covers as well as subvarieties.  In this section we analyze the $a$-invariant for generically finite covers and deduce some consequences for the behavior of rational curves.  

\begin{defi}
Let $X$ be a smooth uniruled projective variety and let $L$ be a big and nef $\mathbb Q$-divisor on $X$.
An $a$-cover of $X$ with respect to $L$ is a generically finite dominant morphism of degree $\geq2$ from a smooth projective variety $f: W\rightarrow X$ such that $a(W, f^*L)= a(X, L)$.

We say that an $a$-cover $f: W\rightarrow X$ is rigid if $(W,f^{*}L)$ is adjoint rigid.
\end{defi}

The geometry of $a$-covers for Fano threefolds has been studied in \cite{LT16}:

\begin{theo}
\label{theo: a-covers}
Let $X$ be a smooth Fano threefold of Picard rank $1$, index $1$, and degree $H^{3} \geq 4$.  Furthermore if $H^{3} = 4$ assume that $-K_{X}$ is very ample.
Assume that $X$ is general in its moduli.
Let $F$ be the Fano scheme of conics on $X$ and $\pi : \mathcal U \rightarrow F$ be the universal family with the evaluation map $s: \mathcal U \rightarrow X$.
Then $s: \mathcal U \rightarrow X$ is an $a$-cover.
Moreover, any $a$-cover $f: Y \rightarrow X$ rationally factors through $s: \mathcal U \rightarrow X$.
\end{theo}

\begin{proof}
It is clear that $s: \mathcal U \rightarrow X$ is an $a$-cover.
Suppose we have an $a$-cover $f: Y \rightarrow X$.
It is shown in \cite[Theorem 1.9]{LT16} that there is no rigid $a$-cover of $X$.
Thus the Iitaka dimension of $-f^*K_X + K_Y$ is positive.
If the Iitaka dimension is $1$, this means that $X$ admits a covering family of surfaces $S$ with $a(S, -K_X)=1$ such that $(S, -K_X)$ is adjoint rigid.
However, it follows from Theorem~\ref{theo: weaklybalanced for Fano 3-folds} that such a family of adjoint rigid surfaces does not exist.
Hence the Iitaka dimension of $-f^*K_X + K_Y$ is $2$ and a general fiber of the Iitaka fibration of $-f^*K_X + K_Y$ is $C\cong \mathbb P^1$ with $a(C, f^*L)=1$.
This means that the image $f(C)$ is a conic and $C\rightarrow f(C)$ is birational.
By the universal property, $f: Y \rightarrow X$ rationally factors through $s: \mathcal U \rightarrow X$.
\end{proof}

\begin{rema}
Note in particular that the proof shows that the ramification divisor for the cover $s: \mathcal{U} \rightarrow X$ has Iitaka dimension $2$.  The ramification divisor is contracted by the map to the parameter space.
\end{rema}

We now introduce a piece of notation we will use frequently in the rest of the paper.  Suppose that $M$ is a component of $\overline{M}_{0,0}(X, d)$ that generically parametrizes birational maps from $\mathbb{P}^{1}$ to its image.  We will let $M^{(n)} \subset \overline{M}_{0,n}(X, d)$ denote the component which generically parametrizes $n$-pointed curves such that the underlying curve is parametrized by $M$. 

The above theorem has the following implication for moduli spaces of rational curves.

\begin{theo} \label{theo: irrfibers}
Let $X$ be a smooth Fano threefold of Picard rank $1$, index $1$, and degree $H^{3} \geq 4$.  Furthermore if $H^{3} = 4$ assume that $-K_{X}$ is very ample.
Assume that $X$ is general in its moduli.
Let $M$ be a reduced dominant component of $\overline{M}_{0,0}(X, d)$ generically parametrizing birational stable maps from $\mathbb P^1$ with $d\geq 3$.
Consider the evaluation map
\[
\mathrm{ev}_1 : M^{(1)} \rightarrow X.
\]
Then for a general $x\in X$, the fiber $\mathrm{ev}_1^{-1}(x)$ is irreducible.
\end{theo}

\begin{proof}
Let $\widetilde{M}^{(1)}$ be a smooth resolution of $M^{(1)}$ and take the Stein factorization
\[
\mathrm{ev}_1: \widetilde{M}^{(1)} \rightarrow Y \rightarrow X.
\]
We denote the intermediate morphisms by $f: Y \rightarrow X$ and $g: \widetilde{M}^{(1)} \rightarrow Y$.

For a contradiction we assume that $f : Y \rightarrow X$ is not birational.
Let $\beta: \widetilde{Y}\rightarrow Y$ be a smooth resolution.
Then it follows from the proof of \cite[Proposition 5.6]{LT17} that $f\circ \beta : \widetilde{Y} \rightarrow X$ is an $a$-cover, hence by the previous theorem, $f \circ \beta$ factors through $s:\mathcal U \rightarrow X$ after possibly replacing $\widetilde{Y}$ by a blow up.
Let $\pi : \widetilde{M}^{(1)\circ} \rightarrow M^\circ$ be the corresponding family of rational curves on $\widetilde{Y}$ over an open subset of $M$.
Then this is dense in the component of $\overline{M}_{0,0}(\widetilde{Y})$ so we conclude that
\[
-K_Y.C = d = \dim M = -(f\circ\beta)^*K_X.C,
\]
where $C$ is a general member of $\widetilde{M}^{(1)\circ}$.
This means that $C$ does not intersect the ramification divisor.  Since the ramification divisor has Iitaka dimension $2$ and is vertical for the map to the moduli space, the image of $C$ in $X$ must be a conic,
and $C\rightarrow X$ is a multiple cover onto a conic.
This contradicts with our assumption that a general member is birational.
\end{proof}

The following corollary is an immediate consequence of the previous theorem (just as in \cite[Observation 5.17]{LT17}).

\begin{coro}
\label{coro: irreducibility}
Let $X$ be a smooth Fano threefold of Picard rank $1$, index $1$, and degree $H^{3} \geq 4$.  Furthermore if $H^{3} = 4$ assume that $-K_{X}$ is very ample.
Assume that $X$ is general in its moduli.
Let $M_1$ be a reduced dominant component of $\overline{M}_{0,0}(X, d_1)$ which parametrizes a dominant family of stable maps which are generically birational maps from $\mathbb P^1$ to its image.
Let $M_2$ be another reduced dominant component of $\overline{M}_{0,0}(X, d_2)$ generically parametrizing stable maps from $\mathbb P^1$.
Suppose that $d_1\geq 3$. Then there is a unique component $N$ of 
\[
M_1^{(1)} \times_X M_2^{(1)}
\]
such that $N$ generically parametrizes a union of two free curves. Furthermore, $N$ has the expected dimension $d_1 + d_2 -1$.
\end{coro}
\begin{proof}
Let $M_i^{\circ}$ be the open locus of $M_i$ parametrizing free curves.  If $M_i^{(1)\circ}$ denotes the preimage if $M_{i}^{\circ}$ in $M_{i}^{(1)}$ then the evaluation map $M_i^{(1) \circ} \rightarrow X$ is flat by \cite[Corollary 3.5.4]{Kollar}. Thus our assertion follows from Lemma~\ref{lemm: glued components}.
\end{proof}

The following lemma describes how free curves glue with lines.

\begin{lemm}
\label{lemm: free+lines}
Let $X$ be a smooth Fano threefold of Picard rank $1$, index $1$, and degree $H^{3} \geq 4$ such that $-K_{X}$ is very ample.  Assume that $X$ is general in its moduli.
Let $M$ be a reduced dominant component of $\overline{M}_{0,0}(X, d)$  which parametrizes a dominant family of stable maps which are generically birational maps from $\mathbb P^1$ to its image.
Then any component $N$ of $M^{(1)} \times_X \overline{M}_{0,1}(X, 1)$ has the expected dimension $d$ and its image under the natural map to $M$ has codimension at most $1$.
\end{lemm}

\begin{proof}
Since $X$ is general in its moduli $\overline{M}_{0,0}(X, 1)$ is a smooth curve. 
Let $Z$ be the surface swept out by lines. Then the evaluation map $\overline{M}_{0,1}(X, 1) \rightarrow Z$ is finite because of \cite[Lemma 2.1.8]{KPS16}. (Note that a general quartic threefold does not contain a cone.)  To analyze the other evaluation map $s: M^{(1)} \rightarrow X$ first consider the pullback $V = s^{-1}(Z)$ which has dimension $d$.
Let $V = \cup_i V_i$ be the irreducible decomposition.
Note that $N$ is a component of $V_i \times_X \overline{M}_{0,1}(X, 1)$ for some $i$, and since the evaluation map from $\overline{M}_{0,1}(X,1)$ is finite every component of this product has the expected dimension.
Since the forgetful map $M^{(1)} \rightarrow M$ is flat, the image of $V_{i}$ in $M$ has codimension at most $1$.  Since $N$ dominates $V_{i}$ our assertion follows.
\end{proof}

\section{Low degree curves} \label{sec: lowdegreecurves}

In this section we collect several results describing low degree rational curves on general Fano threefolds of Picard rank $1$ and index $1$.  Our main goal is to show that in degrees $3$ and $4$ there is only one component of $\Mor(\mathbb{P}^{1},X)$ parametrizing maps that are birational onto their image.

\subsection{Conic curves}

We first recall the behavior of conics on general Fano threefolds of Picard rank $1$ and index $1$.

\begin{theo}
\label{theo: fanoschemeofconics}
Let $X$ be a smooth Fano threefold of Picard rank $1$ and index $1$ such that $-K_{X}$ is very ample and $X$ has degree $H^{3} \geq 4$.
We assume that when $4 \leq H^3 \leq 10$, $X$ is general in its moduli.
Then there is a unique dominant component $\overline{M}_{0,0}(X, 2)$ which is birational to the Fano surface $F(X)$ of conics,
and this component generically parametrizes a smooth free conic in $X$.
Moreover $F(X)$ is smooth.
\end{theo}

\begin{proof}
Let $M$ be a dominant component of $\overline{M}_{0,0}(X, 2)$.
Since lines cannot cover $X$, a general curve on $M$ is a birational free stable map from $\mathbb P^1$ to a conic. This implies that $M$ is birational to a component of $F(X)$.
The irreducibility and smoothness of $F(X)$ can be found in \cite{KPS16} for $H^3\geq 12$ and in \cite{IM07} for $4\leq H^3 \leq 10$.
\end{proof}

\begin{lemm}
\label{lemm: conics}
Let $X$ be a smooth Fano threefold of Picard rank $1$ and index $1$ such that $-K_{X}$ is very ample and $X$ has degree $4 \leq H^{3} \leq 8$.
Let $F$ be the Fano scheme of conics in $X$. Then the Abel-Jacobi map from $F$ to the intermediate Jacobian $\mathrm{IJ}(X)$ is generically finite onto the image.
\end{lemm}

\begin{proof}
Suppose that $H^3=8$ so that $X$ is a complete intersection of three quadrics.
It is shown in \cite{IM07} that the Albanese variety $\mathrm{Alb}(F)$ of $F$ and the intermediate Jacobian $\mathrm{IJ}(X)$ of $X$ are isomorphic to each other. However, if the Abel-Jacobi mapping from $F$ to $\mathrm{IJ}(X)$ is not generically finite onto the image, then it must factor through a smooth curve $C$ of higher genus. Then we have the following factorization:
\[
\mathrm{Alb}(F) \twoheadrightarrow \mathrm{Jac}(C) \rightarrow \mathrm{IJ}(X)
\]
and the composition must be an isomorphism. This implies that $\mathrm{IJ}(X)$ is isomorphic to the Jacobian of $C$, but this is impossible. (See \cite[Theorem 8.1.7]{iskov})

Suppose that $H^3 = 4$. In this case, $X$ is a smooth quartic hypersurface in $\mathbb P^4$.
Our assertion follows from \cite[Remark after the proof of Lemma 1]{Let84}.

Suppose that $H^3 =6$. In this case, $X$ is a smooth complete intersection of a quadric and cubic.
The same proof for $H^3 = 4$ works just fine.
\end{proof}

\subsection{Cubic curves}

Our first goal is to prove that the generic cubic curve on $X$ is a twisted cubic.  

\begin{lemm}
\label{lemm: planecubics2}
Let $X$ be a smooth Fano threefold of Picard rank $1$ and index $1$ such that $-K_{X}$ is very ample and $X$ has degree $H^3 \geq 4$. Let $M$ be a component of $\overline{M}_{0,0}(X, 3)$ parametrizing a dominant family of stable maps whose generic member has irreducible domain and maps birationally onto its image.  Then a general member of $M$ is a stable map to a twisted cubic (and not a rational plane cubic).
\end{lemm}

\begin{proof}
It suffices to show that $M$ does not generically parametrize rational cubic curves contained in a $\mathbb{P}^{2}$.  We will assume otherwise and deduce a contradiction.

Consider the incidence correspondence $I \subset M \times (\mathbb P^{g+1})^*$ which parametrizes pairs $(f: C \rightarrow X, H)$ such that $f(C) \subset H$.  By our assumption $I$ has dimension $g+1$.  Let $p_{2}: I \to \mathbb{P}^{g+1}$ denote the projection map.  We will separate into cases based on the dimension $d$ of the general fiber of $p_{2}$.

\begin{itemize}
\item[$d=0$.] Then $p_{2}$ is dominant.  In particular, a very general hyperplane section of $X$ should contain an irreducible cubic curve.  However, this would violate the Noether-Lefschetz theorem on the Picard group for hyperplane sections of a smooth threefold.
See \cite{Mou67}.

\item[$d=1$.] In this case, a general hyperplane $H$ in the image of $p_{2}$ will contain a one-parameter family of rational curves.  Note  that outside of a codimension $2$ subset of the parameter space hyperplane sections of $X$ have canonical singularities.  Applying adjunction we see that outside of a codimension $2$ subset in the parameter space a hyperplane section cannot be uniruled, giving a contradiction.

\item[$d=2$.]  
In this case, a general hyperplane $S$ in the image of $p_{2}$ will contain a two-parameter family of rational curves.  Let $\phi: \widetilde{S} \to S$ be the resolution of such a surface.  By taking strict transforms, we find a $2$ dimensional family of rational curves $C$ covering $\widetilde{S}$. Since these curves form a dominant family on the smooth surface $\widetilde{S}$ the corresponding component of the moduli space of rational curves on $\widetilde{S}$ has the expected dimension.  Since by assumption this dimension is exactly $2$, we have
\begin{equation*}
-K_{\widetilde{S}} \cdot C - 1 = 2.
\end{equation*}
Since the class of $C$ is nef, we have
\begin{equation*}
0 \leq a(S,\phi^{*}H)\phi^{*}H \cdot C + K_{S} \cdot C =  3a(S,\phi^{*}H) - 3.
\end{equation*}
Altogether this implies that $a(S, H) =1$,
and the Iitaka fibration of $\phi^*H+ K_{\widetilde{S}}$ is fibered by conics.
In this case $C$ should have zero intersection against these conics, a contradiction.

\end{itemize}
\end{proof}

\begin{theo}
\label{theo: cubicsirreducible}
Let $X$ be a smooth Fano threefold of Picard rank $1$ and index $1$ such that $-K_{X}$ is very ample and $X$ has degree $H^3 \geq 4$. Assume that $X$ is general in its moduli. Then there is a unique component $M$ of $\overline{M}_{0,0}(X, 3)$ parametrizing a dominant family of stable maps whose generic member has irreducible domain and maps birationally onto its image.
\end{theo}

\begin{proof}
By Lemma~\ref{lemm: planecubics2} it suffices to prove that the locus $M^{\circ}$ parametrizing twisted cubics is irreducible.  When $H^3 = 22$, this is settled in \cite{KS04}. So we may assume that $4 \leq H^3 \leq 18$. We first explain the argument in the case $4 \leq H^{3} \leq 8$.  Let $N$ denote the irreducible space of twisted cubics in $\mathbb{P}^{g+1}$ and let $\mathcal{C}$ denote the universal curve over $N$.  Consider the diagram
\begin{equation*}
\xymatrix{\mathcal{C} \ar[r]^{f}\ar[d]_{\pi} &  \mathbb{P}^{g+1} \\
N & }
\end{equation*}
Suppose that $X$ is a complete intersection given by hypersurfaces of degree $a_{1}, \ldots, a_{r}$.  Consider $\pi_{*}f^{*}\mathcal{E}$ where $\mathcal{E} = \mathcal{O}(a_{1}) \oplus \ldots \oplus \mathcal{O}(a_{r})$.  We claim that this is a globally generated vector bundle.  The fact that $\pi_{*}f^{*}\mathcal{E}$ is a vector bundle follows from the computation $H^{1}(C,\mathcal{E}|_{C}) = 0$.  The global generation follows from the fact that the map $H^{0}(\mathbb{P}^{n},\mathcal{E}) \to H^{0}(C,\mathcal{E})$ is surjective (which can be deduced easily using the Eagon-Northcott resolution of the ideal sheaf of $C$).  Since $f$ has connected fibers we can identify sections of $f^{*}\mathcal{E}$ with sections of $\mathcal{E}$.  Thus $M^{\circ}$ can be identified with the vanishing locus of a general section of the globally generated vector bundle $\pi_{*}f^{*}\mathcal{E}$ and we deduce that it is irreducible.

The argument in the case $10 \leq H^{3} \leq 18$ is essentially the same.  The only change is that we work with a homogeneous space $\mathbb{G}$ in the place of $\mathbb{P}^{g+1}$.  In particular we must prove a surjection $H^{0}(\mathbb{G},\mathcal{E}) \to H^{0}(C,\mathcal{E})$ for the homogeneous space $\mathbb{G}$.  However, we already have a surjection on sections from the ambient projective space and this map factors through $\mathbb{G}$.  We also need to know that $f$ has connected fibers, but this is still true for a homogeneous space.
\end{proof}

\subsection{Quartic curves}

The argument for quartic curves is very similar to the argument for cubics.

\begin{lemm}
\label{lemm: degeneratequartics}
Let $X$ be a smooth Fano threefold of Picard rank $1$ and index $1$ such that $-K_{X}$ is very ample. When $4 \leq H^3 \leq 8$, we further assume that $X$ is general.  Let $M$ be a component of $\overline{M}_{0,0}(X, 4)$ parametrizing a dominant family of stable maps whose generic member has irreducible domain and maps birationally onto its image.  Then a general member of $M$ is a stable map to a rational normal curve of degree $4$.
\end{lemm}

\begin{proof}
It suffices to show that $M$ does not generically parametrize smooth rational quartic curves contained in a $\mathbb{P}^{3}$.  We will assume otherwise and deduce a contradiction.
For concreteness we will make the further assumption that a general stable map on $M$ spans a $\mathbb P^3$. (When it only spans a $\mathbb P^2$ the argument is much simpler.)  We also assume that the degree of $X$ is $\geq 6$; the case when $X$ has degree $4$ will be discussed at the end.

Consider the incidence correspondence $I \subset M \times (\mathbb P^{g+1})^*$ which parametrizes pairs $(f: C \rightarrow X, H)$ such that $f(C) \subset H$.  By our assumption $I$ has dimension $g+1$.  Let $p_{2}: I \to \mathbb{P}^{g+1}$ denote the projection map.  We will separate into cases based on the dimension $d$ of the general fiber of $p_{2}$.

\begin{itemize}
\item[$d=0$.] Then $p_{2}$ is dominant.  In particular, a very general hyperplane section of $X$ should contain an irreducible quartic curve.  However, this would violate the Noether-Lefschetz theorem on the Picard group for hyperplane sections of a smooth threefold.
See \cite{Mou67}.

\item[$d=1$.] In this case, a general hyperplane $H$ in the image of $p_{2}$ will contain a one-parameter family of rational curves.  Note  that outside of a codimension $2$ subset of the parameter space hyperplane sections of $X$ have canonical singularities.  Applying adjunction we see that outside of a codimension $2$ subset in the parameter space a hyperplane section cannot be uniruled, giving a contradiction.

\item[$d=2$.]  
In this case, a general hyperplane $S$ in the image of $p_{2}$ will contain a two-parameter family of rational curves.  We separate into two cases based on degree.

If $6 \leq H^{3} \leq 8$, then we can repeat the argument for $d=1$ by appealing to Lemma~\ref{lemm: noncanonicallocus_lowdegree} to ensure that the generic hyperplane section in the image of the map has canonical singularities.

If $H^{3} \geq 10$, let $\phi: \widetilde{S} \to S$ be a resolution.  By taking strict transforms, we find a $2$ dimensional family of rational curves $C$ covering $\widetilde{S}$.  Since these curves form a dominant family on the smooth surface $\widetilde{S}$ the corresponding component of the moduli space of rational curves on $\widetilde{S}$ has the expected dimension.  This implies that
\begin{equation*}
-K_{\widetilde{S}} \cdot C - 1 = 2.
\end{equation*}
Since the class of $C$ is nef, we have
\begin{equation*}
0 \leq a(\widetilde{S},\phi^{*}H)\phi^{*}H \cdot C + K_{\widetilde{S}} \cdot C = 4a(\widetilde{S},\phi^{*}H) - 3.
\end{equation*}
Thus we see that $a(\widetilde{S},\phi^{*}H) \geq 3/4$.  As proved in Theorem~\ref{theo: a-value=3/4}
this implies that either $a(\widetilde{S},\phi^{*}H) = 1$ or $10 \leq H^3 \leq 16$ and $S$ is a non-normal hyperplane. Note that the second case contradicts with Lemma~\ref{lemm: normality for H high degree}.
We show that the first case $a(S, H)=1$ is also impossible. By Theorem~\ref{theo: weaklybalanced for Fano 3-folds}, the adjoint divisor $\phi^*H + K_{\widetilde{S}}$ has Iitaka dimension $1$. 
Let $\alpha: \widetilde{S} \rightarrow S'$ be a minimal model for $\phi^*H + K_{\widetilde{S}}$. On this model we have
\[
\alpha_* \phi^* H + K_{S'} = eF
\]
where $F$ is a general fiber of the Iitaka fibration for $\alpha_* \phi^* H + K_{S'}$.
Since $\alpha_* \phi^* H + K_{S'}$ is an integral divisor, we conclude that $e$ is a positive integer.
Let
\[
\phi^*H + K_{\tilde{S}} = eF + E
\]
be the Zariski decomposition where $F$ is a general fiber of the Iitaka fibration of $\phi^*H + K_{\widetilde{S}}$.
Since $(\phi^*H + K_{\widetilde{S}}).C = 1$, this implies that we have $e = 1$.
Thus we conclude that 
\[
H^2.S \leq (\alpha_* \beta^* H)^2 = (F-K_{S'})^2 = 4 + K_{S'}^2.
\]
Since $\widetilde{S}$ is rationally connected, $S'$ is a Hirzebruch surface and thus $H^2. S \leq 12$.
This is impossible when $H^3 \geq 14$.
If $H^3 = 12$, then by \cite[Theorem 5.3]{Kuz05} the Fano surface of conics is the symmetric product of a smooth curve of genus $7$. Thus it contains only finitely many rational curves, and this is a contradiction as the base of the Iitaka fibration of $\phi^*H + K_{\widetilde{S}}$ must be $\mathbb P^1$.
If $H^3 =10$, then it follows from \cite[Corollary 7.3]{DIM12} that the Fano surface of conics only contains finitely many rational curves, a contradiction.

\item[$d = 3$.] 
Suppose that the dimension of the general fiber is $3$.
Fix a general $H \in \pi(I_i)$. By generality, there is a $3$-dimensional family of rational curves of degree $4$ on $S = H \cap X$.
Let $\phi: \widetilde{S}\rightarrow S$ be a resolution.
By taking strict transforms, we obtain a $3$ dimensional family of degree $4$ rational curves covering $\widetilde{S}$.  Since these curves form a dominant family on the smooth surface $\widetilde{S}$ the corresponding component of the moduli space of rational curves on $\widetilde{S}$ has the expected dimension.  This implies that
\[
-K_{\widetilde{S}}.C-1 =  3,
\]
where $C$ is a general member of this family.
Since deformations of $C$ cover $\widetilde{S}$, we also have
\[
0 \leq a(S, H)\phi^*H.C + K_{\widetilde{S}}.C =  4a(S, H)-4.
\]
Altogether this implies that $a(S, H) =1$,
and the Iitaka fibration of $\phi^*H+ K_{\widetilde{S}}$ is fibered by conics.
However, in this case $C$ should have zero intersection against these conics, a contradiction.
\end{itemize}
This finishes the proof when $H^{3} \geq 6$.  When $H^{3}=4$, the argument is essentially the same. Note that if a general member of $M$ is a rational normal quartic curve, then it is not contained in a hyperplane section. Thus we cannot consider the incidence correspondence as above. Instead, we may assume that a general member of $M$ is contained in $\mathbb P^3$. Then one can consider the incidence correspondence as above and deduce a contradiction.
\end{proof}

\begin{theo}
\label{theo: quarticsirreducible}
Let $X$ be a smooth Fano threefold of Picard rank $1$ and index $1$ such that $-K_{X}$ is very ample and $X$ has degree $4 \leq H^3 \leq 18$.  Assume that $X$ is general in its moduli.  Then there is a unique component $M$ of $\overline{M}_{0,0}(X, 4)$ parametrizing a dominant family of stable maps whose generic member has irreducible domain and maps birationally onto its image.
\end{theo}

\begin{proof}
By Lemma \ref{lemm: degeneratequartics} we see that any component of $\overline{M}_{0,0}(X,4)$ as in the statement will generically parametrize normal rational curve of degree $4$. Thus it suffices to show that the space of normal rational quartics on $X$ is irreducible.  This follows from the same arguments as Theorem \ref{theo: cubicsirreducible}.
\end{proof}

\section{Movable Bend and Break Lemma}
\label{sec: movable}

Recall that $Z \subset X$ denotes the divisor swept out by lines.  Recall also that if $M$ is a component of $\overline{M}_{0,0}(X, d)$ that generically parametrizes birational maps from $\mathbb{P}^{1}$ to its image then $M^{(n)} \subset \overline{M}_{0,n}(X, d)$ denotes the corresponding $n$-pointed parameter space.

\begin{lemm}
\label{lemm: veryfree}
Let $X$ be a smooth Fano threefold of Picard rank $1$ and index $1$ such that $-K_{X}$ is very ample.
Furthermore when the degree satisfies $4 \leq H^3 \leq 8$ we assume that $X$ is general in its moduli.
Let $d\geq 4$ and let $M$ be a component of $\overline{M}_{0,0}(X, d)$ such that a general curve $f: C \rightarrow X$ in $M$ is irreducible, birational onto its image,  and has image not contained in $Z$.
Then for a general $f: C \rightarrow X$ in $M$, $f$ is very free. 
\end{lemm}

\begin{proof}
Since the image of $C$ is not contained in $Z$, a general curve $f: C \rightarrow X$ in $M$ must be free by Theorem \ref{intro:expecteddim}. We only need to show that
\[
\mathrm{ev}_2 : M^{(2)} \rightarrow X \times X
\]
is dominant.
Suppose that it is not dominant.  Since there is at least a one-parameter family of curves through a general point, the image must be an irreducible divisor $D$ in $X^2$.

Choose a general point $(x_1, x_2) \in D$ so that a general curve $C$ in $M$ passing through $x_1, x_2$ is a free birational stable map and the dimension of $\mathrm{ev}_2^{-1}(x_1, x_2)$ is  $d-3$. 
Pick an irreducible component $T$ of $\mathrm{ev}_1^{-1}(x_1)$ which is $d-2$ dimensional, and let $S$ be the surface swept out by free birational curves parametrized by $T$. Choose an embedded resolution $\widetilde{S} \subset \widetilde{X}$ with a birational morphism $\beta: \tilde{X}\rightarrow X$. By blowing up if necessary, we may assume that the locus $\beta^{-1}(x_1)$ is divisorial in $\widetilde{X}$. Strict transforms $\widetilde{C}$ on $\widetilde{S}$ are parametrized by an open locus $T^\circ$ of $T$ which is $d-2$ dimensional, and they meet with the locus $\beta^{-1}(x_1)$ positively.
Any curves in a component of $\overline{M}_{0,0}(\widetilde{S})$ containing $\widetilde{C}$ satisfy the same intersection properties as $\widetilde{C}$, so their pushforwards contain $x_1$. This means that $T^\circ$ is dense in a reduced component of $\overline{M}_{0,0}(\widetilde{S})$. Hence we have
\[
-K_{\widetilde{S}}.\widetilde{C}-1 = d-2.
\]
Since these curves are covering $\widetilde{S}$, we have the inequality
\[
a(S, H)d = a(S, H)\beta^*(H) . \widetilde{C}\geq-K_{\widetilde{S}}.\widetilde{C} 
\]
and this implies that
\[
\frac{d-1}{d} \leq a(S, H).
\]
By Lemma \ref{lemm: spectrumforsurfaces} we can write $a(S, H) = \frac{c}{n}$ where $c= 2$ or $3$ and $n$ is an integer.

First suppose that the degree $d$ is $\geq 5$.  The only possible value of $a(S,H)$ satisfying the above restriction is $1$ and so we must show that $a(S, H)=1$ is impossible.  Using the existence of a 2-parameter family of rational curves on $S$ and arguing just as we did in the $d=2$ case of the proof of Lemma \ref{lemm: degeneratequartics}, we see that $H^{2}.S \leq 12$ and that $S$ admits a rational fibration by conics over the base $\mathbb{P}^{1}$.  We conclude that this is impossible when $10 \leq H^{3}$ just as we did in the proof of Lemma \ref{lemm: degeneratequartics}.
When $4 \leq H^3 \leq 8$, then by generality the Fano surface $F$ contains only finitely many rational curves by Lemma~\ref{lemm: conics}.  Thus $X$ can not be covered by such surfaces, yielding a contradiction in this case.

Suppose now that $d = 4$.  \cite[Main Theorem and Theorem 5.2]{CR17} shows that a general quartic rational curve is very free when the degree is $4 \leq H^3\leq 8$. So we may suppose that $H^3 \geq 10$. Repeating the argument above, our assumption shows that an irreducible $2$-dimensional family of quartics passing through a general point sweeps out an irreducible surface $S$ and either $a(S, H) = 1$ or $a(S,H) = 3/4$.  The first case is ruled out in a similar way.  So it suffices to consider the case when $a(S,H) = 3/4$.  For a general point $x_{3}$ in a general such surface $S$ the family of quartics on $S$ passing through $x_3$ is $2$-dimensional by construction.   However, Theorem~\ref{theo: 34classification} implies that if we take the strict transform of this family of quartics to a resolution of $S$ they must coincide with the strict transform of the family of lines on $\mathbb{P}^{2}$ under a birational map.  In particular there should be only a $1$-dimensional family of such curves through $x_3 \in S$, giving a contradiction.
\end{proof}

\begin{lemm}
\label{lemm: evaluation}
Let $X$ be a smooth Fano threefold of Picard rank $1$ and index $1$ such that $-K_{X}$ is very ample and $X$ has degree $4 \leq H^{3} \leq 18$.
Assume that $X$ is general in its moduli.
Let $d \geq 2$ and let $M$ be a component of $\overline{M}_{0,0}(X, d)$ such that a general curve $C$ is an irreducible birational stable map and not contained in $Z$. Then 
\[
\mathrm{ev}_n : M^{(n)} \rightarrow X^n
\]
is dominant if and only if $d\geq 2n$.
\end{lemm}
\begin{proof}
Since a general curve $C$ in $M$ is free, the component $M$ has the expected dimension $d$. Thus $d\geq 2n$ is a necessary condition for $\mathrm{ev}_n$ to be dominant.
When $d =2, 3$, then our assertion is clear because a general curve is free.  

We need to consider the case when $d \geq 4$.
By Lemma \ref{lemm: veryfree} any general curve in $M$ is very free so in particular a general $f: C\rightarrow X$ is a closed immersion into $X$. 
If we look at the sublocus of $M$ parametrizing free curves through $n$ fixed points, the tangent space and the obstruction space of $M$ at $C$ are given respectively by
\[
H^0(C, N_{C/X}(-n)) \qquad \textrm{and} \qquad H^{1}(C,N_{C/X}(-n)).
\]
\cite[Theorem 1.4]{Shen10} shows that $N_{C/X} = \mathcal O(a) \oplus \mathcal O(b)$ satisfies $|a-b| \leq 1$.  For a curve of degree $d$ the degree of $N_{C/X}$ is $d-2$, so if we fix $\lfloor d/2 \rfloor$ points the obstruction space vanishes.  
This means that if there exists a curve through $\lfloor d/2 \rfloor$ points, then the space of such curves has the expected dimension (namely $0$ if $d$ is even or $1$ if $d$ is odd).  Consider the incidence correspondence which parametrizes free curves in our family with balanced normal bundle and $(\lfloor d/2 \rfloor)$-tuples of distinct points in $X$ such that the curve contains all the points.  Every non-empty fiber of the projection onto $\mathrm{Sym}^{\lfloor d/2 \rfloor}(X)$ has the expected dimension; a dimension count then shows that the projection onto $\mathrm{Sym}^{\lfloor d/2 \rfloor}(X)$ is dominant.  This proves our statement.
\end{proof}

\begin{coro}
\label{coro: domains}
Let $X$ be a smooth Fano threefold of Picard rank $1$ and index $1$ such that $-K_{X}$ is very ample and $X$ has degree $4 \leq H^3 \leq 18$.
Assume that $X$ is general in its moduli.
Let $d \geq 2$ and $M$ be a component of $\overline{M}_{0,0}(X, d)$ such that
\[
\mathrm{ev}_n : M^{(n)} \rightarrow X^n
\]
is dominant where $n = \lfloor d/2 \rfloor$.  Then $M$ generically parametrizes stable maps which are irreducible birational.
\end{coro}

\begin{proof}
We prove this statement by induction on $d$.
The case of $d=2$ is settled by Theorem~\ref{theo: fanoschemeofconics}.
For the induction step, first we assume that $d = 2n +1$ is odd.
Suppose that $M$ generically parametrizes stable maps whose domains are reducible.
We denote a general stable map by $f: C_1\cup C_2 \rightarrow X$ passing through $n$ general points.
However, since the image contains $n$ general points, if we denote the degrees of the curves by $2m_1, 2m_2 +1$ so that $n=m_1 + m_2$ then we see $C_1$ must contain $m_{1}$ general points and $C_2$ must contain $m_{2}$ general points. 
By the induction hypothesis, each $C_i$ is an irreducible birational stable map.
However, by Corollary \ref{coro: irreducibility} and Lemma~\ref{lemm: free+lines}, such stable maps cannot form a $d$-dimensional component of the moduli space, a contradiction.
Thus the domain of a general stable map in $M$ must be irreducible.
If a general stable map in $M$ is a multiple cover, then for dimension reasons, it must be a multiple cover of a conic. But clearly such curves can not contain the desired number of general points.
The case of $d=2n$ is similar.
\end{proof}

\begin{coro} \label{coro: genpointsclassification}
Let $X$ be a smooth Fano threefold of Picard rank $1$ and index $1$ such that $-K_{X}$ is very ample and $X$ has degree $4 \leq H^3 \leq 18$.
Assume that $X$ is general in its moduli.
Fix $n \geq 1$ general points of $X$ and let $f: C \to X$ be a stable map whose image contains all $n$ points.
\begin{enumerate}
\item Suppose that $f$ has degree $2n$.  Then $f$ has irreducible domain and maps birationally onto a free curve.
\item Suppose that $f$ has degree $2n+1$.  If the image of $f$ is reducible, then either it is a union of  free curves or the union of a free curve and a line.
\item Suppose that $f$ has degree $2n+2$.  If the image of $f$ is reducible, then either it is a union of free curves, the union of an irreducible free curve with an irreducible curve of degree $\leq 2$, or the union of a line with one free curve of even degree and one free curve of odd degree, or the union of an irreducible free curve with two lines that do not intersect.
\end{enumerate}
\end{coro}

\begin{proof}
We begin the proof with an observation we will use many times.  Fix a curve $C$ in $X$.  Then a general member of a family of free curves of degree $d$ can not intersect $C$.  Indeed, standard deformation theory shows that the locus parametrizing free curves through any fixed point of $X$ must have codimension at least $2$ in the moduli space, so the locus parametrizing free curves through a curve must have codimension at least $1$.

(1) Suppose that the image of $f$ is a union of two (possibly reducible) curves $C_{1},C_{2}$.  Using Lemma \ref{lemm: evaluation} we see that both $C_{i}$ must have even degrees $2m_{i}$ with $m_{1}+m_{2} = n$ and that each $C_{i}$ must contain $m_{i}$ of these general points.  By induction on the degree both $C_{1}$ and $C_{2}$ must be irreducible.  Furthermore, for each $i$ the family of curves we obtain as we vary the general points is dense in the corresponding component of $\overline{M}_{0,0}(X)$.  Since there are only finitely many curves of degree $2m$ through $m$ general points, for general choices of points $C_{1}$ and $C_{2}$ can not meet to form a connected curve.  Indeed, by the observation above the incidence correspondence identifying deformations of $C_{1}$ and $C_{2}$ which intersect does not dominate the product of the two moduli components. 
Finally, any rational curve of bounded degree through a general point of $X$ must be free.

(2) Suppose that the image of $f$ is a union of irreducible curves $C_{1},C_{2},\ldots,C_{k}$.  By Lemma \ref{lemm: evaluation}, at most one of these curves can have odd degree, say $C_{k}$.  Then every other $C_{i}$ must have even degree $2m_{i}$ and go through $m_{i}$ of the points; in particular there are only finitely many possibilities for each and as we vary the points the curves are dense in the corresponding component of $\overline{M}_{0,0}(X)$.  Furthermore each is a free curve.  The remaining curve $C_{k}$ must have odd degree $2m_{k}+1$ and go through $m_{k}$ of the points.  In order to obtain a connected curve, $C_{k}$ also must meet each of the $C_{i}$.

Now suppose that $C_{k}$ is not free, so that $m_{k}=0$.  In other words, $C_{k}$ is a line.  We claim that there can be only one other component $C_{1}$.  Once we fix a component $C_{1}$ containing a general point, there are only finitely many lines intersecting it.  By applying the observation at the beginning of the proof, we see that a general deformation of any free curve will not intersect either $C_{1}$ or this finite set of lines.  Thus to obtain a connected curve we can not allow any more components.

(3) Suppose that the image of $f$ is a union of irreducible curves $C_{1},C_{2},\ldots,C_{k}$.  There are two cases to consider.

First suppose that each $C_{i}$ has even degree.  By Lemma \ref{lemm: evaluation}, there is at most one component $C_{k}$ not containing any of the general points, and if such a component exists it has degree $2$ (as a stable map).  Then every other $C_{i}$ has even degree $2m_{i}$ and goes through $m_{i}$ of the points; in particular there are only finitely many possibilities for each and each is a free curve.  In order to obtain a connected curve, $C_{k}$ also must meet each of the $C_{i}$.  If the map onto $C_{k}$ is a double cover of a line, then by arguing as in (2) there is only one other component.  Otherwise the map onto $C_{k}$ is a birational map onto a non-free conic.  Since non-free conics sweep out a surface, we can argue just as we did for lines to see that there is only one other component.

Second suppose that some $C_{i}$ has odd degree.  By Lemma \ref{lemm: evaluation} there are then exactly two such components $C_{k-1}, C_{k}$ of degrees $2m_{k-1}+1$ and $2m_{k}+1$ respectively.  Every other $C_{i}$ has even degree $2m_{i}$ and goes through $m_{i}$ of the points; in particular there are only finitely many possibilities for each and each is a free curve.  By Lemma \ref{lemm: evaluation} $C_{k-1}$ must contain $m_{k-1}$ of the general points and $C_{k}$ must contain $m_{k}$ of the points.  Furthermore, no two of the $C_{i}$ can meet for $i < k-1$.

Suppose that $m_{k} = 0$ but $m_{k-1} > 0$.  Then $C_{k}$ is a line.  We show that there are then at most three components.  Either $C_{k-1}$ and $C_{k}$ meet or they do not.  Suppose that $C_{k-1}$ and $C_{k}$ meet.  If we fix $m_{k-1}$ general points, there is a $1$-parameter family of curves $C_{k-1}$ through the points, and each $C_{k-1}$ intersects finitely many lines $C_{k}$.  If we choose the components $C_{1},\ldots,C_{k-2}$ general, then by the observation above each will intersect with only finitely many members of this one-parameter family.  Thus to obtain a connected curve we see that there can be at most one other component $C_{1}$.  Suppose that $C_{k-1}$ and $C_{k}$ do not meet.  By generality $C_{k}$ and $C_{k-1}$ both only meet with one of the free curves $C_{i}$ (using similar arguments as before) and to obtain a connected curve this component must be the same for each.

Suppose that both $m_{k-1}$ and $m_{k}$ are $0$ so both $C_{k-1}$ and $C_{k}$ are lines.  Note that two general lines cannot meet by \cite[Lemma 2.1.8]{KPS16}. (As mentioned before, a general quartic threefold does not contain a cone.)  Thus by arguing as above we see that there is only one other component which must be a free curve.
\end{proof}

\begin{theo}{\textnormal{(Movable Bend and Break lemma)}} \label{theo: movbendandbreak}
Let $X$ be a smooth Fano threefold of Picard rank $1$ and index $1$ such that $-K_{X}$ is very ample and $X$ has degree $4 \leq H^3 \leq 18$.  Assume that $X$ is general in its moduli.
Let $d \geq4$ and let $M$ be a reduced dominant component of $\overline{M}_{0,0}(X, d)$ generically parametrizing birational stable maps from $\mathbb P^1$. Then $M$ contains a stable map from a chain of length $2$ where each component maps birationally onto a free curve.  
\end{theo}

\begin{theo} \label{theo: countingcomponents}
Let $X$ be a smooth Fano threefold of Picard rank $1$ and index $1$ such that $-K_{X}$ is very ample and $X$ has degree $4 \leq H^3 \leq 18$.  Assume that $X$ is general in its moduli.
Let $d \geq 3$.  Then there is a unique component $M$ of $\overline{M}_{0,0}(X, d)$ which is dominant and generically parametrizes birational stable maps from $\mathbb P^1$.  
\end{theo}

\begin{proof}[Proof of Theorem \ref{theo: movbendandbreak} and Theorem \ref{theo: countingcomponents}:]
We prove both statements simultaneously using induction on $d$.
The base case for Theorem \ref{theo: movbendandbreak} is when $d=4$.
This follows from Theorem~\ref{theo: quarticsirreducible}.
The base case for Theorem \ref{theo: countingcomponents} is when $d=3, 4$.  
This follows from  Theorem~\ref{theo: cubicsirreducible} and Theorem~\ref{theo: quarticsirreducible}. So from now on we will suppose $d \geq 5$.

Suppose both statements are true in degrees less than $d$. For $2\leq e < d$, we denote the unique dominant component generically parametrizing birational stable maps from $\mathbb P^1$ of degree $e$ by $M_e$. We record a useful consequence of the induction hypothesis for later applications.  Since the Picard rank of $X$ is $1$, every free curve will intersect every divisor on $X$.  Thus, if we fix a component $M$ of $\overline{M}_{0,0}(X,e)$ which generically parametrizes free curves, there is a closed set of codimension $\geq 2$ in $X$ which contains all points not contained in any free curve parametrized by $M$.  We let $W$ denote the union of all such closed subsets as we vary $e < d$.  Then for any degree $f$ with $3 < f < d$ and for any point $x \in X \backslash W$ there are two free curves which both contain $x$ whose degrees add up to $f$.  Applying Theorems \ref{theo: irrfibers} and \ref{theo: countingcomponents}, we see that for any point $x \in X \backslash W$ any free curve through $x$ of degree $f$ with $3<f<d$ deforms into a chain of two free curves while keeping $x$ fixed.

Now we verify that Theorem \ref{theo: movbendandbreak} holds in degree $d$.
Suppose that $d=2n+1 \geq 5$ is odd.
Then it follows from Lemma~\ref{lemm: evaluation} that the evaluation map
\[
M^{(n)} \rightarrow X^n
\]
is dominant.  Choose $n$ general points: there is a one parameter family of irreducible curves through such points.  By bend and break these curves break into a stable map whose image is reducible.  Applying Corollary \ref{coro: genpointsclassification}.(2) we see that either all the components are free or the domain has two components $C_{1}, C_{2}$ where (without loss of generality) the image of $C_{1}$ is a degree $d-1$ free curve and the image of $C_{2}$ is a general line.

We start with the latter case.  We apply the observation at the beginning of the proof to note that $C_{1}$ breaks into free curves $C_{1}' \cup C_{1}''$ while keeping the intersection point with the line $C_{2}$ fixed.  Since $C_{2}$ is a general line, its normal bundle is given by
\[
N_{C_2/X} = \mathcal O \oplus \mathcal O(-1).
\] 
In particular, the stable map onto the curve $C_{1}' \cup C_{1}'' \cup C_{2}$ is a smooth point of $\overline{M}_{0,0}(X, d)$.  This means that $M$ must contain any irreducible locus which contains this stable map.  In particular, after deforming one of the free curves we may suppose this stable map is a chain; let $C_{1}''$ be the component which intersects the line.

Note that the subchain given by $C_{1}'' \cup C_{2}$ is also a smooth point of a component $N$ of $\overline{M}_{0,0}(X)$.  By a dimension count we see that chains of this type are not dense in $N$.  Since also the number of components can not go down under specialization, we see that the generic point of $N$ parametrizes a map from an irreducible domain onto a free curve.  Since $C_{1}'$ is free, we can move it along with a deformation in $N$ to see that our special stable map is a deformation of a chain of two free curves.  The argument above shows that this chain must be contained in $M$.

In the other case, every component of the deformed curve is free.  This implies that the corresponding stable map is a smooth point of $M$.
Since the deformed curve passes through general $n$ points, there is no multiple cover among the components.  We can separate the domain of the stable map into two connected components of degree $e$ and $f$ respectively with $e+f=d$.  The restriction of the stable map to these two connected components yield stable maps in $M_{e}^{(1)}$ and $M_{f}^{(1)}$.  By \cite[II.7.6 Theorem]{Kollar} each of these stable maps can be smoothed while keeping the intersection point fixed.  Since our original stable map was a smooth point of $M$, we see that $M$ contains the union of a free curve of degree $e$ with a free curve of degree $f$ and our assertion follows.

Next suppose that $d = 2n\geq 6$ is even.
Then it follows from Lemma~\ref{lemm: evaluation} that the evaluation map
\[
M^{(n)} \rightarrow X^n
\]
is dominant and generically finite.
Choose $n-1$ general points and a general free conic $D$,
and consider a general curve passing through these $n-1$ points and meeting with $D$.
Such curves $C$ form a $1$ dimensional family, so by bend and break these curves break into a stable map whose image is the union of two connected curves $C_1 \cup C_2 \rightarrow X$.  Applying Corollary \ref{coro: genpointsclassification}.(3) we obtain a finite list of possibilities.  Note that the deformed curve can not be the union of an irreducible degree $d-2$ free curve and two lines, the union of an irreducible degree $d-2$ free curve with a double line, the union of an irreducible degree $d-2$ free curve with a non-free conic, or the union of a line with a free curve of even degree and a free curve in odd degree. 
Indeed, in all cases the arguments of Corollary \ref{coro: genpointsclassification} show that there are only finitely many such curves through a set of $n-1$ general points.  Thus such curves can not meet a general conic by the observation in the proof of Corollary \ref{coro: genpointsclassification}.  The only remaining possibilities are that every component of the deformed curve is free, or the curve is the union of an irreducible free curve and a line.  In both cases we can argue as above to conclude the statement.

We next verify that Theorem \ref{theo: countingcomponents} in degree $< d$ and Theorem \ref{theo: movbendandbreak} in degree $\leq d$ together imply Theorem \ref{theo: countingcomponents} in degree $d$.  Let $M$ be any component of $\overline{M}_{0,0}(X, d)$ which is dominant and generically parametrizes birational stable maps from $\mathbb P^1$.  By Theorem \ref{theo: movbendandbreak} $M$ parametrizes a union of two free curves $C_e \cup C_f$ of degree $e, f$ with $e+f=d$ and $e \leq f$. This is a smooth point of $\overline{M}_{0,0}(X,d)$, thus we conclude that $M$ contains the main component of $M_e^{(1)} \times_X M_f^{(1)}$. Note that by Corollary~\ref{coro: irreducibility} this main component is unique.

For the final step, we will break and glue curves to show that $M$ contains the main component of $M_{2}^{(1)} \times_{X} M_{d-2}^{(1)}$.  Suppose that $3 \leq e$.  Applying the induction hypothesis and Theorem \ref{theo: irrfibers} we see that $M_{e}$ contains a point parametrizing a union of a general conic and a general curve of degree $e-2$.  For general choices the corresponding chain of three curves will be a smooth point of the moduli space.  Again using the induction hypothesis and the fact that the number of components of a stable map can not go down under specialization, we see that $M$ must contain the main component of $M_2^{(1)} \times_X M_{d-2}^{(1)}$.  There can only be one such component, proving uniqueness.
\end{proof}

Finally we prove our main theorem.

\begin{proof}[Proof of Theorem \ref{theo: maintheorem}:]
Although the statement assumes that the anticanonical degree is $\geq 3$, it is a bit cleaner to consider all families of rational curves at once.

First suppose that $M \subset \Mor(\mathbb{P}^{1},X)$ is a component which generically parametrizes maps that are birational onto their image.  If $M$ is dominant, then it must be the unique dominant family constructed by Theorem \ref{theo: countingcomponents} when the degree is $\geq 3$ or the unique family of conics when the degree is $2$.

If $M$ is not dominant, then \cite[Theorem 1.1]{LT17} shows that the curves parametrized by $M$ sweep out a surface $Z \subset X$ with $a(Z,-K_{X}|_{Z}) > 1$.  \cite[Section 6.4]{LTT14} shows that any such surface must be swept out by lines.  Since $X$ is general in moduli, the parameter space of lines is a smooth curve of genus $\geq 2$.  This implies that there is a unique surface $Z$ with $a(Z,-K_{X}|_{Z}) > 1$ and that the only rational curves on $Z$ are the lines.  Thus $M$ must be the component of $\Mor(\mathbb{P}^{1},X)$ parametrizing lines.

Now suppose that $M$ generically parametrizes maps that are $d:1$ onto their images.  Let $C$ denote a rational curve that is the image of a general map parametrized by $M$.  Suppose that $N$ is the parameter space for morphisms $\mathbb{P}^{1} \to X$ which are birational onto the deformations of $C$.  Then
\begin{equation*}
\dim(M) = \dim(N) + 2d-2.
\end{equation*}
Note that the difference of the expected dimensions of $M$ and $N$ is $(d-1)\cdot (-K_{X} \cdot C)$.  Since $M$ must have at least the expected dimension, we see that either $N$ must parametrize curves of anticanonical degree at most $2$ or $N$ must have larger than the expected dimension.  These yield exactly the two cases in the statement of the theorem.
\end{proof}

\section{Geometric Manin's Conjecture}
\label{sec: GM}

Let $X$ be a smooth projective Fano variety and set $L=-K_{X}$.  Manin's Conjecture predicts that the asymptotic growth of the number of components of $\Mor(\mathbb{P}^{1},X)$ of $L$-degree $d$ is controlled by $a(X,L)$ and $b(X,L)$.  We emphasize that we are only interested in the ``absolute case'' of rational curves on a fixed variety $X$.  Often Manin's Conjecture is phrased in the ``relative case'' as a count of sections of a family over a base curve $B$; from this perspective the ``absolute case'' means that we count sections of the trivial family $X \times \mathbb{P}^{1}$ over the base curve $\mathbb{P}^{1}$.

Just as in Manin's Conjecture for rational points, one should discount an ``exceptional set'' consisting of curves with higher growth rates than expected.  Here we recall the basic outline of the theory from \cite[Section 6]{LT17}.  For simplicity we will only describe the theory in the case when $-K_{X} \cdot C = 1$ for some curves $C$ on $X$ (as this is the only case we will need).

For a component $W$ of $\Mor(\mathbb{P}^{1},X)$ we let $\mathcal{C}_{W} \to W$ denote the universal family of rational curves over $W$.  We define $M_{d} \subset \Mor(\mathbb{P}^{1},X)$ to be the union of components $W$ satisfying the following conditions:
\begin{enumerate}
\item $W$ parametrizes rational curves of $L$-degree $d$.
\item The morphism $\mathcal{C}_{W} \to X$ does not factor through any morphism $f: Y \to X$ such that $\dim(Y) < \dim(X)$.
\item The morphism $\mathcal{C}_{W} \to X$ does not factor through any morphism $f: Y \to X$ such that $f$ is dominant and $\kappa(K_{Y} + a(Y,L)f^{*}L) > 0$.
\item The morphism $\mathcal{C}_{W} \to X$ does not factor through any morphism $f: Y \to X$ such that $f$ is dominant and face contracting and
\begin{equation*}
(a(X,L),b(X,L)) \leq (a(Y,f^{*}L),b(Y,f^{*}L))
\end{equation*}
in the lexicographic order.
\end{enumerate}
The definition of face contracting can be found in \cite[Section 3.3]{LT17}.  In our situation, it translates to the condition that $f$ induces a strict inequality of $a,b$-invariants in the lexicographic order.

Fix a formal variable $q$.  We define the counting function
\begin{equation*}
\overline{N}(X,q,d) = \sum_{i=1}^{d} \sum_{W \in M_{i}} q^{\dim W}
\end{equation*}
which records the dimensions and number of components of $M_{i}$ for $i \leq d$.  Manin's Conjecture predicts that
\begin{equation*}
\overline{N}(X,q,d) \sim Cq^{d}
\end{equation*}
where $C$ is an explicit constant determined by the geometry of the nef cone of curves.

In the situation of Theorem \ref{theo: maintheorem} we have seen that $M_{d}$ consists of a single component of dimension $d+3$ in every sufficiently large degree, yielding the asymptotic formula
\begin{equation*}
\overline{N}(X,q,d) \sim \frac{q^{3}}{1-q^{-1}} q^{d}.
\end{equation*}
This formula agrees with the predictions of \cite{LT17}.

\section{Gromov-Witten theory} \label{sec: gw}

In this section we use our results to show that certain Gromov-Witten invariants on $X$ are enumerative.  Consider the moduli spaces $\overline{M}_{0,n}(X,d)$.  In each degree $\geq 4$, there is a unique ``main'' component $M_{main} \subset \overline{M}_{0,n}(X,d)$ that generically parametrizes $n$-pointed very free rational curves of degree $d$.

We will focus on the pointed invariants $\langle [pt]^{n} \rangle_{0,n}^{X,2n}$.  To show these are enumerative, it is enough to observe the following:
\begin{itemize}
\item By Corollary \ref{coro: genpointsclassification}, the intersection of the pullback of point classes from the $n$ evaluation maps will vanish along every component of $\overline{M}_{0,n}(X,2n)$ except for $M_{main}$.
\item Since $M_{main}$ generically parametrizes maps that are birational onto free curves, this component does not admit a generic stabilizer and is generically smooth.
\end{itemize}
Together, these results show that the entire contribution of the virtual fundamental class to the Gromov-Witten invariant comes from the fundamental class of $M_{main}$.  Furthermore the generic smoothness guarantees that our intersection numbers count actual curves in a general situation.

\begin{rema}
By appealing to the classification of Corollary \ref{coro: genpointsclassification}, one can prove in a similar way the enumerativity of GW-invariants of the form $\langle [pt]^{n}, C \rangle_{0,n+1}^{X,2n+1}$ where $C$ is a general member of a dominant family of curves on $X$.
\end{rema}

There are many results in the literature relating the Gromov-Witten invariants of a complete intersection to the Gromov-Witten invariants of the ambient space; see for example \cite{gathmann02}.  In particular the $1$- and $2$-point Gromov-Witten invariants for Fano $3$-folds of Picard rank $1$ have been computed in \cite{PRZY} and \cite{GOLYSHEV}.  The other Gromov-Witten invariants can be recovered using the recursive formula \cite[Theorem 5.2]{BK05}.




\bibliographystyle{alpha}
\bibliography{Fano3-folds}

\end{document}